\pdfoutput=1
\documentclass{article}

\usepackage{amsmath,amsthm,enumitem,multicol,amssymb,mathrsfs,verbatim,graphicx,bbm}

\usepackage{geometry,xcolor}
\usepackage[makeroom]{cancel}
\usepackage[normalem]{ulem}
\makeatletter
\def\uwave{\bgroup \markoverwith{\lower3.5\p@\hbox{\sixly \textcolor{red}{\char58}}}\ULon}
\font\sixly=lasy6
\makeatother

\usepackage{tikz}
\usetikzlibrary{arrows}

\usepackage{lscape}

\DeclareMathSymbol{\R}{\mathbin}{AMSb}{"52}
\newcommand{\A}{\mathcal{A}}

\newcommand{\E}{\mathcal{E}}
\newcommand{\motzn}{\mathcal{M}_n}
\newcommand{\U}{\mathcal{F}}

\renewcommand\mp{\ensuremath{M}}

\newcommand{\Res}{\mathop{\rm Res}\nolimits}
\newcommand{\Ine}{\mathop{\rm Ine}\nolimits}
\newcommand{\Edif}{\mathop{\rm Edif}\nolimits}

\def\dd{\kern.1ex\mbox{\raise.7ex\hbox{{\rule{.35em}{.15ex}}}}\kern.1ex}
\def\sdd{\kern.1ex\mbox{\raise.4ex\hbox{{\rule{.35em}{.15ex}}}}\kern.1ex}

\def\orsg{\mathop{\sf Orsg}\nolimits}
\def\lsg{\mathop{\sf lsg}\nolimits}
\def\rsg{\mathop{\sf rsg}\nolimits}
\def\des{\mathop{\sf des}\nolimits}

\def\exce{\mathop{\sf exc}\nolimits}
\def\wex{\mathop{\sf wex}\nolimits}
\def\cross{\mathop{\sf cross}\nolimits}
\newcommand\fix{{\sf fix}}

\def\inv{\mathop{{\sf inv}}\nolimits}
\def\maj{\mathop{{\sf maj}}\nolimits}

\def\occ{\mathop{\sf occ}\nolimits}

\newcommand{\wt}{\mathop{\sf wt}\nolimits}

\newcommand{\exc}{\mathop{\sf exc}\nolimits}
\newcommand{\lexc}{\mathop{\sf le}\nolimits}
\newcommand{\aexc}{\mathop{\sf aexc}\nolimits}
\newcommand{\laexc}{\mathop{\sf lae}\nolimits}

\newcommand{\nile}{\mathop{\sf nile}\nolimits}
\newcommand{\ile}{\mathop{\sf ile}\nolimits}
\newcommand{\ie}{\mathop{\sf ie}\nolimits}
\newcommand{\nie}{\mathop{\sf nie}\nolimits}
\newcommand{\nilae}{\mathop{\sf nilae}\nolimits}
\newcommand{\ilae}{\mathop{\sf ilae}\nolimits}
\newcommand{\iae}{\mathop{\sf iae}\nolimits}
\newcommand{\niae}{\mathop{\sf niae}\nolimits}
\newcommand{\fp}{\mathop{\sf fp}\nolimits}
\newcommand{\iefp}{\mathop{\sf iefp}\nolimits}

\newcommand{\inve}{\mathop{\sf inve}\nolimits}
\newcommand{\ninve}{\mathop{\sf ninve}\nolimits}
\newcommand{\inva}{\mathop{\sf inva}\nolimits}
\newcommand{\prex}{\mathop{\sf prex}\nolimits}
\newcommand{\fola}{\mathop{\sf fola}\nolimits}
\newcommand{\ninva}{\mathop{\sf ninva}\nolimits}

\newtheorem{lemma}{Lemma}
\newtheorem{corollary}{Corollary}
\newtheorem{theorem}{Theorem}
\newtheorem*{theorem*}{Main Theorem}
\newtheorem{proposition}{Proposition}
\newtheorem{conjecture}{Conjecture}

\theoremstyle{definition}

\newtheorem{remark}{Remark}
\newtheorem{definition}{Definition}
\newtheorem{problem}{Problem}

\newcommand\ZZ{\mathbb Z}
\newcommand\NN{\mathbb N}
\newcommand\clsn{\mathcal{S}_n}
\newcommand\snk{\mathcal{S}_n^k}

\newcommand\sig{\sigma}

\newcommand\cf{\ensuremath{\mathcal{C}}}
\newcommand\cfu{\ensuremath{\mathcal{C_\bu}}}

\newcommand{\params}{\ensuremath{a,b,c,d,f,g,h,\ell,p,r,s,t,u,w}}
\newcommand{\gz}{\ensuremath{\mathcal C_{\params}(z)}}

\newcommand\karr{$k$-arrangement}
\newcommand\karrs{$k$-arrangements}
\renewcommand{\a}{\ensuremath{\mathfrak{a}}}

\newcommand\cd{\cdot}
\newcommand\ra{\rightarrow}

\newcommand\fixpts{{\mathsf{Fix}}}

\newcommand\ank{{\mathsf{A_n^k}}}
\newcommand\arr[2]{{\mathsf{A_{#2}^{{#1}}}}}

\def\csummand{\ensuremath{
a^{\ile(\sigma)} 
b^{\nile(\sigma)} 
c^{\ie(\sigma)-\ile(\sigma)}   
d^{\nie(\sigma)-\nile(\sigma)}   
f^{\ilae(\sigma)} 
g^{\nilae(\sigma)} 
h^{\iae(\sigma)-\ilae(\sigma)}
\ell^{\niae(\sigma)-\nilae(\sigma)}}\\[-1.5ex]&&\hspace{1pt}\times\,
p^{\exc(\sigma)-\lexc(\sigma)} 
r^{\aexc(\sigma)-\laexc(\sigma)} 
s^{\lexc(\sigma)} 
t^{\laexc(\sigma)} 
u^{\fp(\sigma)}  
w^{\iefp(\sigma)}
z^n 
}

\def\wtprod{\ensuremath{
a^{\ile(\sigma)} 
b^{\nile(\sigma)} 
c^{\ie(\sigma)-\ile(\sigma)}   
d^{\nie(\sigma)-\nile(\sigma)}   
f^{\ilae(\sigma)} 
g^{\nilae(\sigma)} 
h^{\iae(\sigma)-\ilae(\sigma)}
\ell^{\niae(\sigma)-\nilae(\sigma)}}\\[0.2ex]&&\hspace{1pt}\times\,
p^{\exc(\sigma)-\lexc(\sigma)} 
r^{\aexc(\sigma)-\laexc(\sigma)} 
s^{\lexc(\sigma)} 
t^{\laexc(\sigma)} 
u^{\fp(\sigma)}  
w^{\iefp(\sigma)}
}

\newcommand\disp{\displaystyle}

\title{Permutations, Moments, Measures}

\author{Natasha Blitvi\'c\thanks{This work was completed while the author was an Academic Guest at ETH Z\"urich.}\\
\small Department of Mathematics and Statistics\\[-0.8ex]
\small Lancaster University, UK\\[-0.8ex]
\small\tt natasha.blitvic@lancaster.ac.uk\\
\and
Einar Steingr\'imsson\thanks{Partially supported by a Leverhulme Research Fellowship.}\\
\small Department of Mathematics and Statistics\\[-0.8ex]
\small University of Strathclyde, UK\\[-0.8ex]
\small\tt einar@alum.mit.edu}

 \date{}

\begin{document}

\maketitle
\thispagestyle{empty}

\begin{abstract}
Which combinatorial sequences correspond to moments of probability measures on the real line?  We present a generating function, in the form of a continued fraction, for a fourteen-parameter family of such sequences and interpret these in terms of combinatorial statistics on the symmetric groups. 
Special cases include several classical and noncommutative probability laws, along with a substantial subset of the orthogonalizing measures in the $q$-Askey scheme, now given a 
new combinatorial interpretation in terms of elementary permutation statistics. 
This framework further captures a variety of interesting combinatorial sequences including, notably, the moment sequences associated to distributions of the 
numbers of occurrences of (classical and vincular) permutation patterns of length three. This connection between pattern avoidance and broader ideas in classical and noncommutative probability is among several intriguing new corollaries, which generalize and unify results 
previously appearing in the literature, while opening up new lines of inquiry. 

The fourteen combinatorial statistics further generalize to signed and colored 
permutations, and, as an infinite family of statistics, to the \emph{$k$-arrangements}: permutations with $k$-colored fixed points, introduced here along with several related results and conjectures.

\end{abstract}

\section{Introduction}\label{section-intro}

By describing the distributions of random objects in terms of their moment sequences, a number of fundamental probability laws are seen to be equivalent to key constructions in combinatorics. A well-known example is the semicircle law \cite{Wigner1955} whose ubiquity in random matrix theory and free probability is paralleled by the pervasiveness of its moments, the Catalan numbers, in the enumerative world. 
Similarly, the moments of the Poisson and Gaussian laws enumerate, respectively, set 
partitions and perfect matchings of sets, while the `noncommutative analogues' of these probability laws are given by combinatorial refinements of the aforementioned moment sequences \cite{Bozejko1991, Anshelevich2001, Blitvic2012, Bozejko2015, Ejsmont}. This probabilistic interpretation further extends to combinatorial statistics on set partitions, symmetric groups and other Coxeter groups \cite{Bozejko1994,Bozejko2017}. 

The classical theorem of Hamburger \cite{Hamburger} provides a complete characterization of moment sequences associated to (positive Borel) measures on the real line. Namely, given some real-valued sequence $(m_n)_{n\geq 0}$, the existence of a measure $\mu$ with the property that
\begin{equation}\label{eq-moments-general}
\int_\R x^n\,d\mu(x)=m_n
\end{equation}
is equivalent to the positive semidefiniteness of the Hankel matrices $(m_{i+j})_{0\leq i,j \leq n}$ for all $n\in\NN_0.$ (The measure $\mu$ is of course a probability measure whenever $m_0=1$.)  Unfortunately, establishing the positivity of the Hankel matrices is a difficult combinatorial problem. Rather, showing that a given combinatorial sequence is a moment sequence is typically approached by other means, including explicitly writing down the corresponding measure (see e.g.\ \cite{Mlotowski2010Documenta,Mlotowski2012Colloquium,Mlotowski2014,Martin2015,Mlotowski2018,Sokal2018}).

On the whole, the integer sequences that are positive definite (in the preceding sense) are relatively few, but their class seems to include surprisingly many classical sequences. 
Rather than considering such problems in isolation, we are interested in general principles that may help elucidate the link between combinatorial structure and positivity. 

In this paper, we introduce a fourteen-parameter class of combinatorial sequences which are moments of measures on the real line.  Our central object is the following continued fraction. 
\begin{definition}\label{def-gz}
For parameters $\params\in\R$, let
\begin{equation}\label{eq-main}
\cf(z)=\gz:=\cfrac{1}{1-\alpha_0z-\cfrac{\beta_1z^2}{1-\alpha_1z-\cfrac{\beta_2z^2}{\ddots}}}
\end{equation}
with
\begin{equation}\label{eq-alpha-beta}
\alpha_n=u\cd w^n + s\,[n]_{a,b} + t\,[n]_{f,g},\quad \quad
\beta_n=p\,r\,[n]_{c,d}\,[n]_{h,\ell}.
\end{equation}
\end{definition}
(Above, we follow the standard convention by letting $[n]_{x,y}:=x^{n-1}+x^{n-2}y+\cdots+xy^{n-2}+y^{n-1}$ for $n\in\NN$, with $[0]_{x,y}:=0$.)

By our main theorem (Theorem~\ref{thm-main}, next section), the continued fraction $\cf$ has a natural combinatorial interpretation in terms of fourteen elementary combinatorial statistics on permutations, to be defined in detail in the next section.  Except for the traditionally defined fixed points, these are all based on \emph{excedances}, that is, 
integers $i\in[n]:=\{1,2,\ldots,n\}$ such that $\sig(i)>i$, \emph{anti-excedances}, $\sig(i)<i$, and \emph{inversions} among letters of $\sig$, that is, pairs $(i,j)$ such that $i<j$ and $\sig(i)>\sig(j)$.  We distinguish between \emph{linked} and \emph{non-linked} excedances (an intuitive geometric property, defined shortly) and anti-excedances, respectively, as well as between inversions and non-inversions associated to these. 
An aspect of the continued fraction $\cf$ worth highlighting is the systematic nature of its combinatorial interpretation. Namely, the fourteen combinatorial statistics characterize the action of $\sigma\in\clsn$ on each element of $[n]:=\{1,2,\ldots,n\}$ by placing each pair $i\mapsto \sigma(i)$ into one of the five aforementioned excedance-based classes, with the relations between the elements in each class (and in one case between classes) captured in terms of inversions and non-inversions. As a result, it becomes relatively straightforward to understand, on a combinatorial level, the appearance in specializations of $\cf$ of the combinatorial families given in Table~\ref{table-measures}. 

Continued fraction expansions of ordinary generating functions commonly appear in combinatorial theory, particularly in relation to labeled Motzkin paths \cite{Flajolet1980}. As such, the continued fraction $\cf$ generalizes aspects of several important papers connecting continued fractions, Motzkin paths and statistics on permutations and other combinatorial structures, such as \cite{fravie,Flajolet1980,foata-zeilberger-denerts,biane-exc-inv-heine,medicis-viennot-laguerre,elizalde-cf-fpsac}.  Given a continued fraction associated with a well-studied combinatorial sequence, there are typically two approaches to arriving at a more general result.  One consists of attempting to `fold in' multiple related combinatorial statistics into one generating function, such as in e.g.  \cite{ClarkeSteingrimssonZeng,ismail-kasraoui-zeng-orthogonal,elizalde-cf-fpsac}. That is, given some multivariate generating function associated to several combinatorial sequences, the problem consists of identifying the continued fraction expansion of the said generating function.  Most recently, Sokal and Zeng~\cite{sokal-zeng-masterpolys} have undertaken an impressive study of various ways in which statistics on permutations, set partitions and perfect matchings can be consolidated in this manner.

The other approach, taken presently, seeks to provide a more general \emph{scheme} of generating functions. Here, the challenge is that of giving a natural combinatorial interpretation for some multiparameter continued fraction or, equivalently, for a multiparameter family of polynomials that are orthogonal with respect to some linear functional. The continued fractions studied may be associated with new schemes of orthogonal polynomials, such as \cite{simion-stanton-laguerre,simion-stanton-octabasic,randrianarivony-poly-ortho-sheffer}, or well-known ones, as in the particularly notable works of Corteel and Williams et al.~\cite{CorteelWilliams1, CorteelWilliams2} that provide a combinatorial interpretation for the Askey-Wilson family of $q$-hypergeometric orthogonal polynomials.

We find that the continued fraction $\cf$ of Definition~\ref{def-gz} includes a surprising number of interesting specializations of the Askey-Wilson polynomials, which appear as named families in the \emph{$q$-Askey scheme} \cite{koekoek-hyp-ortho-poly-q}. In this 
manner, $\cf$ ascribes a new and considerably more transparent combinatorial interpretation, in terms of elementary permutation statistics, to a substantial portion of the $q$-Askey scheme, while generalizing these families in a different direction.  Furthermore, $\cf$ extends the so-called octabasic Laguerre family of Simion and Stanton~\cite{simion-stanton-laguerre,simion-stanton-octabasic}, also associated with combinatorial statistics on the symmetric groups.  A crucial difference between our approach and that of Simion and Stanton is in the explicit accounting of the fixed points, which in prior work are combined with anti-excedances to form non-excedances. The result is a considerably greater degree of generality, admitting a wider array of interesting special cases, and a greater symmetry in the continued fraction. (Note that Simion and Stanton do not work with `excedance based' statistics, but their statistics correspond, via a certain bijection, to such statistics, as further discussed in Section~\ref{sec-comparisons}.)  Our work also partially extends a sixteen-parameter framework of Randrianrivony~\cite{randrianarivony-poly-ortho-sheffer}, itself another generalization of the work of Simion and Stanton. A different combinatorial interpretation of our fourteen-parameter continued fraction \cf\ can be obtained by specializing an
eighteen-parameter one in a recent paper by Sokal and Zeng~\cite{sokal-zeng-masterpolys}. They accomplish this by `mixing' statistics of two inherently different kinds, based respectively on excedances and `records' (relative maxima and minima), and by also using refinements of the crossing and nesting statistics of Corteel~\cite{Corteel2007}.
A detailed discussion of the similarities and differences between all these combinatorial approaches is found in Section~\ref{sec-comparisons}.

The simplicity of our combinatorial interpretation of the continued fraction \cf\ particularly pays dividends when considering the corollaries and extensions of the main theorem (Theorem~\ref{thm-main}). In particular, simple generalizations of our continued fraction capture corresponding statistics on the hyperoctahedral groups (signed permutations) and certain unitary reflection groups (colored permutations). Statistics on the latter were introduced in~\cite{es-indexed} and subsequently studied by various authors. In another direction, the separation of the fixed points in $\cf$ leads naturally to a generalization capturing various statistics on the \emph{$k$-arrangements}, introduced here, namely permutations with $k$-colored fixed points. These form an infinite family of combinatorial objects containing in a natural way the derangements (fixed-point free permutations), permutations and what previously have been called just arrangements by Comtet~\cite{comtet-advanced}, which also coincide with Postnikov's definition of `decorated permutations'~\cite{postnikov-webs}.  We present formulas and an exponential generating function for the numbers of \karrs\, along with a few results on combinatorial statistics on these objects and several conjectures.

Most of the corollaries appearing in this paper are new. Those that have previously appeared in more specialized works gain clarity by being seen to follow from a more general principle, as well as a new dimension, through juxtaposition to new ideas. This 
general perspective also opens up new lines of inquiry. As previously mentioned, the combinatorial statistics naturally lead us to define the notion of a $k$-arrangement, which come with interesting properties, some of which are shown here and others conjectured. This combinatorial framework also raises a number of natural questions from the point of view of probability and analysis, as the measures arising in the present setting include a surprising number of well-studied probability laws which were previously not connected into a single framework. And at the interface of probability and combinatorics, the framework suggests that pattern avoidance, at least for some patterns, may be naturally interpreted in terms of moments of random variables (whether classical or noncommutative), proposing new ways of approaching the long-standing open problems in this arena.

\begin{centering} ---------\\ \end{centering}

This paper is organized as follows:  In Section~\ref{sec-def-central-thm} we define the 
permutation statistics at the core of our constructions and present the main theorem  (Theorem~\ref{thm-main}), which describes the connection between these statistics and the continued fraction \cf. The correspondence relies on labeled Motzkin paths, using Flajolet's general correspondence.  

In Section~\ref{section-results} we describe, in six subsections, the various results derived from the main theorem. Namely, we identify the parameter ranges for which a sequence arising from \cf\ is a moment sequence of a unique probability measure on the real line (Section~\ref{subsec-permutation-moment}). As previously mentioned, the family of moment sequences encompassed by \cf\ is surprisingly interesting, including several well-known classical laws as well as noncommutative central limits. Furthermore, the orthogonal polynomials associated with \cf\ include a substantial portion of the $q$-Askey scheme (Section~\ref{sec-AW}). 
From the combinatorial perspective, the Hankel matrices associated to sequences obtained by specializations of~\cf\ all have determinants that are products of squares of factorials, encompassing at least ten known such instances from the literature (Section~\ref{subsec-hankel}). Specializations of \cf\ include a variety of well-known 
combinatorial sequences associated to permutations and set partitions (including perfect matchings), and also the distribution of the numbers of occurrences of several permutation patterns of length 3. As a first substantiated link between permutation patterns and moment sequences, we show that a sequence of numbers of avoiders of a pattern 
of length 3 is the moment sequence of a probability measure if and only if it is obtained as a specialization of \cf\ (Section~\ref{subsec-patterns}). The continued fraction \cf\ also captures a host of statistics on the symmetric groups (permutations), the hyperoctahedral groups (signed permutations), and, more generally, certain types of unitary reflection groups (colored permutations) (Section~\ref{subsec-hyperocta}). The combinatorial framework presented here leads us, in a natural way, to the definition of \emph{$k$-arrangements}, permutations with $k$-colored fixed points, which include the derangements, the permutations and what had previously been called just arrangements. We present some results on the structure and enumeration of these and several conjectures (Section~\ref{subsec-arrangements}).

Finally, in Section~\ref{section-bijection} we give a proof for the central bijection of the paper, between permutations and Motzkin paths labeled to capture the fourteen different 
statistics on permutations, and explain the connections and differences between the results in this paper and those by several other authors.

\section{Definitions and the Main Theorem}
\label{sec-def-central-thm}

The following definitions are most naturally understood using two-row diagrams depicting permutations as bijections from the set $[n]:=\{1,2,\ldots,n\}$ to itself (see Figure~\ref{fig:ex_bij}). We also use one-line notation, presenting a permutation $\sig$ as the word $\sig(1)\sig(2)\ldots \sig(n)$. 

The continued fraction \cf\ in Definition~\ref{def-gz} has a combinatorial interpretation in terms of fourteen elementary combinatorial statistics on permutations, 
defined as follows (see Figure~\ref{fig:ex_bij} for an example):

\begin{definition}\label{def-invstats}
Let $\sig$ be a permutation in $\clsn$ and $[n]=\{1,2,\ldots,n\}$. For each $\sig$ we define its
\begin{enumerate}
\item\label{def-1-1} number of excedances as $\exc(\sig):=\#\{i \in [n]\mid i<\sig(i)\}$,
\item number of fixed points as $\fp(\sig):=\#\{i \in [n]\mid i=\sig(i)\}$,
\item\label{def-1-3} number of anti-excedances as $\aexc(\sig):=\#\{i \in [n]\mid i>\sig(i)\}$,
\item number of linked excedances as $\lexc(\sig):=\#\{i \in [n]\mid \sig^{-1}(i)<i<\sig(i)\}$,
\item number of linked anti-excedances as $\laexc(\sig):=\#\{i \in [n]\mid \sig^{-1}(i)>i>\sig(i)\}$.
\end{enumerate}
We say that \emph{$i$ is an excedance} if $i<\sig(i)$, and likewise for the other definitions above. We also define the following statistics for each $\sig$:

\begin{enumerate}
\setcounter{enumi}{5}
\item\label{firststat} The number of inversions between excedances: 
$\ie(\sig):=\#\{i,j \in [n]\mid i<j<\sig(j)<\sig(i)\}$.

\item The number of inversions between excedances where the greater excedance is linked:\\
$\ile(\sig):=\#\{i,j \in [n]\mid i<j<\sig(j)<\sig(i) ~\mbox{and}~ \sig^{-1}(j)<j\}$.

\item\label{secondstat}  The number of restricted non-inversions between excedances:\\ $\nie(\sig):=\#\{i,j \in [n]\mid i<j<\sig(i)<\sig(j)\}$.

\item The number of restricted non-inversions between excedances where the rightmost excedance is linked: $\nile(\sig):=\#\{i,j \in [n]\mid i<j<\sig(i)<\sig(j) ~\mbox{and}~ \sig^{-1}(j)<j\}$.

\item\label{thirdstat} The number of inversions between anti-excedances:\\
$\iae(\sig):=\#\{i,j \in [n]\mid j>i>\sig(i)>\sig(j)\}$. 

\item The number of inversions between anti-excedances where the smaller anti-excedance is linked:\\
$\ilae(\sig):=\#\{i,j \in [n]\mid j>i>\sig(i)>\sig(j) ~\mbox{and}~ \sig^{-1}(i)>i\}$.

\item\label{fourthstat} The number of restricted non-inversions between anti-excedances:\\
   $\niae(\sig):=\#\{i,j \in [n]\mid j>i>\sig(j)>\sig(i)\}$.

\item The number of restricted non-inversions between anti-excedances where the smaller anti-excedance is linked: $\nilae(\sig):=\#\{i,j \in [n]\mid j>i>\sig(j)>\sig(i) ~\mbox{and}~ \sig^{-1}(i)>i\}$.

\item\label{laststat} The number of inversions between excedances and fixed points:\\ 
$\iefp(\sig):=\#\{i,j \in [n]\mid i<j=\sig(j)<\sig(i)\}$.
\end{enumerate}
\end{definition}

\newcommand\nineperm{
\filldraw [black] 
(1,0) circle (3pt)
(2,0) circle (3pt)
(3,0) circle (3pt)
(4,0) circle (3pt)
(5,0) circle (3pt)
(6,0) circle (3pt)
(7,0) circle (3pt)
(8,0) circle (3pt)
(9,0) circle (3pt)
(1,2) circle (3pt)
(2,2) circle (3pt)
(3,2) circle (3pt)
(4,2) circle (3pt)
(5,2) circle (3pt)
(6,2) circle (3pt)
(7,2) circle (3pt)
(8,2) circle (3pt)
(9,2) circle (3pt);
\node at (1,-0.5) {1};
\node at (2,-0.5) {2};
\node at (3,-0.5) {3};
\node at (4,-0.5) {4};
\node at (5,-0.5) {5};
\node at (6,-0.5) {6};
\node at (7,-0.5) {7};
\node at (8,-0.5) {8};
\node at (9,-0.5) {9};
\node at (1,2.5) {1};
\node at (2,2.5) {2};
\node at (3,2.5) {3};
\node at (4,2.5) {4};
\node at (5,2.5) {5};
\node at (6,2.5) {6};
\node at (7,2.5) {7};
\node at (8,2.5) {8};
\node at (9,2.5) {9};
}

\begin{figure}\centering
\begin{tikzpicture}[
scale=.7,
shorten > = .5ex, 
thick ]
\nineperm
\draw [densely dashed, thick, color=blue] (4,2) -- (4,0);
\draw [dotted, thick, color=red] (7,2) -- (7,0);
\draw [->, color = red] (1,2) -- (5,0);
\draw [->, color = red] (2,2) -- (9,0);
\draw [->, color = red] (3,2) -- (7,0);
\draw [->, color = blue] (4,2) -- (1,0);
\draw [->, color = blue] (5,2) -- (2,0);
\draw [->, color = green] (6,2) -- (6,0);
\draw [->, color = red] (7,2) -- (8,0);
\draw [->, color = blue] (8,2) -- (4,0);
\draw [->, color = blue] (9,2) -- (3,0);
%
%
\begin{scope}[xshift=35em, scale=.8, shorten > = 0ex, ]
\draw[color=black, thin] (-0.5,0) -- (11,0);
\draw[color=red] (0,0) -- (3,3);
\draw[densely dashed, color=blue] (3,3) -- (4.5,3);
\draw[color=blue] (4.5,3) -- (5.5,2);
\draw[color=green] (5.5,2) -- (7,2);
\draw[dotted, color=red] (7,2) -- (8.5,2);
\draw[color=blue] (8.5,2) -- (10.5,0);
\filldraw [black] 
(0,0) circle (2pt)
(1,1) circle (2pt)
(2,2) circle (2pt)
(3,3) circle (2pt)
(4.5,3) circle (2pt)
(5.5,2) circle (2pt)
(7,2) circle (2pt)
(8.5,2) circle (2pt)
(9.5,1) circle (2pt)
(10.5,0) circle (2pt);
\node at (0.3,0.7) {$p$};
\node at (1.0,1.7) {$dp$};
\node at (1.8,2.7) {$cdp$};
\node at (3.8,3.5) {$g^2t$};
\node at (5.5,2.9) {$\ell^2r$};
\node at (6.3,2.4) {$w^2u$};
\node at (7.8,2.4) {$as$};
\node at (9.5,1.7) {$hr$};
\node at (10.3,0.7) {$r$};
\end{scope}
\end{tikzpicture}
\caption{\label{fig:ex_bij}The permutation 597126843
  and the corresponding labeled Motzkin path. Excedances are red with dotted red line indicating linked excedance, anti-excedances blue with dashed blue line indicating linked anti-excedance, fixed point green. Note that the parameters not present in the labels in this example simply have exponent 0, such as $as$ standing in for $ab^0s$.}
\end{figure}
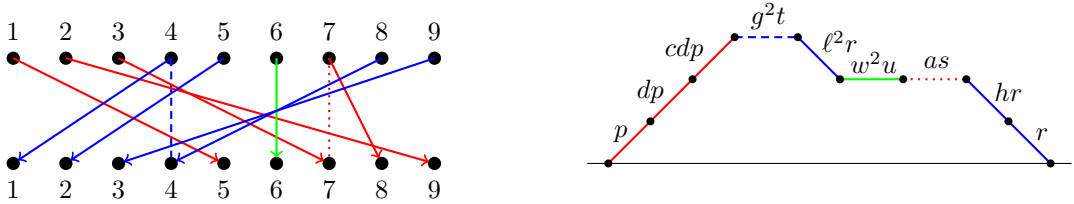

Note that $i$ is a linked excedance if and only if $i<\sig(i)$ and the $j$ for which $\sig(j)=i$ is also an excedance, that is, $j<\sig(j)$.  An analogous statement is true for linked anti-excedances and this is illustrated in Figure~\ref{fig:ex_bij}. Although we do not define this separately, it should be clear that $\exc-\lexc$ counts non-linked excedances, and $\aexc-\laexc$ non-linked anti-excedances.  

The items \ref{firststat}--\ref{laststat} in Definition~\ref{def-invstats} all refer to 
\emph{inversions} and \emph{non-inversions}, the latter with certain restrictions.  A pair $(i,j)$ with $i<j$ forms an inversion if $\sig(i)>\sig(j)$ and a non-inversion if $\sig(i)<\sig(j)$. A non-inversion is \emph{restricted} if we additionally require $\sig(i)>j$ when dealing with excedances or $\sig(j)<i$ in the case of anti-excedances.  Restricted non-inversions are the natural counterparts to inversions when working within the class of excedances (or anti-excedances), as seen by comparing items \ref{firststat} to \ref{secondstat} (resp.  \ref{thirdstat} to  \ref{laststat}) in Definition~\ref{def-invstats}.

The inversions are further refined based on whether the participating excedances or anti-excedances are linked. For example, $\ile(\sig)$ counts the number of inversions between excedances in which the greater excedance is linked. In the permutation 597126843 of Figure~\ref{fig:ex_bij}, the inversion $(2,7)$ is included in this statistic, as $\sig(2)=9>8=\sig(7)$ and the excedance 7 is linked.
As another example with the same permutation, $\niae(\sig)$ is the number of non-inversions between anti-excedances, such as $(4,8)$, where $\sig(4)=1<4=\sig(8)$.  Note that this non-inversion also contributes to $\nilae$ since the smaller anti-excedance 4 is linked.  
Finally, $\iefp$ is the sum over all fixed points $j$ of the number of excedances $i$ preceding $j$ that are greater than $j$ and thus form an inversion with $j$.  In our running example the only fixed point, 6, has two such excedances preceding it, namely 2 and 3.

The continued fraction $\cf$ in Definition \ref{def-gz} gives, in the sense of formal power series, the generating function for all the statistics above. The following is the central theorem of this paper:

\begin{theorem}\label{thm-main}
  \begin{equation}\label{eq-comb-main}
    \begin{array}{lll}
\cf(z)=&\displaystyle\sum_{n\geq 0}\sum_{\sigma_\in \clsn}&
\csummand.
    \end{array}
  \end{equation}
\end{theorem}

Theorem~\ref{thm-main} follows from Flajolet's correspondence \cite[Section~1.1]{Flajolet1980} between continued fractions and labeled Motzkin paths, together with the following definition and theorem.  

Recall that a Motzkin path is a polygonal path in $\R^2$, consisting of upsteps $(1,1)$, level steps $(1,0)$ and downsteps $(1,-1)$, starting at the origin and ending on the $x$-axis without ever going below it (see Figure~\ref{fig:ex_bij}).

\begin{definition}\label{def-paths}
Let $\motzn:=\motzn(\params)$ be the set of Motzkin paths of length~$n$ labeled
 as follows, where $I_k=\{0,1,\ldots,k-1\}$:
 \begin{itemize}
 \item Upsteps from height $k-1$ to height $k\geq 1$ have labels in $\{pc^i d^{k-1-i}\mid i\in I_k\}.$

\item Downsteps from height $k$ to height $k-1$ have labels in 
$\{rh^i \ell^{k-1-i}\mid i\in I_k\}.$ 

\item Level steps at height $k$ have labels in
$$
\{u\cd w^i\mid i\in I_k\}~\cup~\{s\, a^i b^{k-1-i}\mid i\in I_k\}\;~\cup~\{t\, f^i g^{k-1-i}\mid i\in I_k\}.
$$ 
\end{itemize}
The \emph{weight} of a path $\mp \in \motzn$, denoted $\wt(\mp)$, is the product of its labels.
\end{definition}

The following theorem will be proved in Section~\ref{section-bijection}.
\begin{theorem}\label{thm-bij}
There exists a bijection $\eta: \clsn\to \motzn$ with the property that if $\mp=\eta(\sig)$ then
\begin{eqnarray*}
\wt(\mp)\!\!\!&=&
\wtprod.
\end{eqnarray*}
\end{theorem}

The fourteen permutation statistics listed above are basic building blocks from which a number of classical combinatorial statistics on permutations and set partitions can be constructed. Thus, 
by specializing the values of its parameters, the continued fraction $\cf=\gz$ enumerates a variety of familiar objects, examples of which are listed in Table~\ref{table-measures}. Further specializations appear in various results in Sections~\ref{subsec-patterns}--\ref{subsec-arrangements}.

\section{Corollaries and Extensions}
\label{section-results}

The central result of this paper, Theorem~\ref{thm-main}, gives a combinatorial interpretation for the continued fraction \cf\ in terms of fourteen elementary statistics 
on the symmetric groups. This unifying lens gives rise to a variety of corollaries and extensions, described at present.

\subsection{Permutation Statistics as Moment Sequences}
\label{subsec-permutation-moment}

Consider the sequence
\begin{equation}
m_0=1,\quad m_n=[z^n]\gz\,\,\,\,\,\, (n\in\NN),
\label{eq-def-main}
\end{equation}
with the corresponding $\alpha_n,\beta_n$ as in \eqref{eq-alpha-beta} and the parameters $a$ through $w$ real numbers. 

From the general theory of moment problems and continued fractions (e.g. \cite{Shohat}), $(m_n)$ is the sequence of moments of some probability measure $\mu$ on the real line if and only if $\beta_n\geq 0$ for all $n\in\NN$. If $\beta_n>0$ for all $n<n_0$ and $\beta_{n_0}=0$, the measure $\mu$ is unique (up to equivalence) and is supported on a set of $n_0$ elements. If $\beta_n$ is strictly positive for all $n$, there may be multiple non-equivalent measures, supported on infinite sets, whose moments are given by \eqref{eq-def-main}. 

A sufficient condition for the determinacy of the Hamburger moment problem is Carleman's condition (e.g. \cite{Shohat}), namely,
$\sum_{n=1}^\infty \beta_n^{-1/2} =\infty.$
We therefore immediately obtain the following.

\begin{corollary}\label{cor-moments}
  For $\params\in \R$ with $pr>0$ and $c,d,h,\ell$ satisfying
\begin{equation}
c=-d\quad\text{\bf or} \quad h=-\ell \quad\text{\bf or} \quad (c>-d ~\text{ and }~ h> -\ell) \quad\text{\bf or}\quad (c<-d \text{ and } h< -\ell),
\label{eq-parameters-positive}
\end{equation}
the sequence $(m_n)$ in \eqref{eq-def-main} is the moment sequence of some probability measure on the real line. If $p=0$ or $r=0$, the measure is unique and consists of one atom; if $c=-d$ or $h=-\ell$, the measure is unique and is formed of two atoms. More generally, when $\max(|c|,|d|)\cdot\max(|h|,|\ell|)\leq 1,$ the measure is unique.
\end{corollary}

In other words, the continued fraction \cf\ 
encodes moments of a very general class of probability measures on the real line whose moments have combinatorial interpretations in terms of permutations and set partitions. 
Table~\ref{table-measures} lists examples of combinatorial sequences, obtained as specializations of \eqref{eq-def-main}, which are moment sequences of familiar measures. These include several fundamental laws in classical and noncommutative probability, such as various noncommutative central limits. Table~\ref{table-ops} provides further examples, 
associated with sixteen named orthogonal polynomial sequences belonging to the $q$-Askey 
scheme, discussed in more detail in the following section.

That  \cf\  encodes such an abundance of well-studied probability measures is somewhat surprising. What is more, several of the measures in Table~\ref{table-measures} share a 
curious property: as `noncommutative analogues' of the Poisson and Gaussian laws, they are 
in some (noncommutative, combinatorial)  sense stable. 

Permutations and set partitions feature strongly in the basic constructions of noncommutative probability (see e.g. the monograph \cite{NicaSpeicher} for a combinatorial perspective on the free probability theory), but whether combinatorial statistics on permutations and set partitions also carry probabilistic meaning is generally less clear. For instance, the passage from perfect matchings to non-crossing perfect matchings is, from a probabilistic perspective, the difference between the classical Central Limit Theorem (CLT) to the free CLT (e.g. \cite{NicaSpeicher}). Moreover, perfect matchings refined for the number of crossings or those additionally refined for the number of nestings correspond, respectively and in an analogous manner, to the CLT governing a mixture of commuting/anti-commuting elements \cite{Speicher1992} or the CLT describing a mixture of elements that commute with respect to a real-valued commutation coefficient \cite{Blitvic2012}. 

Overall, the existence of a relatively simple unifying combinatorial framework that encodes a broad swath of important probability laws is intriguing. It begs the question of whether the measures belonging to the family \cf\ can be given a similarly consistent probabilistic interpretation, potentially through the lens of noncommutative probability.

\subsection{A scheme of orthogonal polynomials}
\label{sec-AW}

Closely related to continued fractions are sequences of polynomials that are orthogonal with respect to some linear functional.  Recall that any measure $\mu$ on the real line with finite moments, $m_n=\int_\R x^n\,d\mu(x)<\infty$ for all $n\in\NN$, can be associated with a continued fraction.  In particular, from the combinatorial point of view, it is natural to consider the formal power series generating function $\sum_{n\geq0}m_nz^n$, which has a continued fraction expansion of the form~\eqref{eq-main}, with coefficients $\alpha_n\in\mathbb R$ and $\beta_n\geq 0$ ($n\in\NN$) uniquely determined by $\mu$. In turn, 
the polynomials defined by 
\begin{equation}
P_0(x)=1,\quad P_1(x)=x-\alpha_0,\quad P_{n+1}(x)=(x-\alpha_n)P_n(x)-\beta_nP_{n-1}(x)
\label{eq-OP}
\end{equation}
form an orthogonal basis of $L^2(\mu)$, the Hilbert space of real or complex-valued functions that are square-integrable with respect to $\mu$. The practical importance of a convenient orthogonal basis cannot be overstated and, indeed, orthogonal polynomials appear in many areas of pure and applied mathematics, especially in analysis, mathematical physics, and numerical methods. The `classical' families, such as the Hermite, Laguerre, or Charlier (associated, respectively, with the Gaussian, Exponential, and Poisson measures), and their various $q$-analogues are brought together in 
the $q$-\emph{Askey scheme} \cite{koekoek-hyp-ortho-poly-q}.

Moments of the measures associated with hypergeometric orthogonal polynomials have been extensively studied in the combinatorial literature. For instance, while the moments of the orthogonalizing measures for the Hermite, Laguerre, and Charlier polynomials count, respectively, perfect matchings, permutations, and set partitions (see Table~\ref{table-measures}), their $q$-analogues correspond to combinatorial refinements of these structures (see e.g. \cite{ismail-stanton-viennot-comb-hermite-askey, medicis-stanton-white-charlier, simion-stanton-laguerre,simion-stanton-octabasic,kasraoui-stanton-zeng-salam-chihara-laguerre}). The aforementioned families can all be expressed as specializations of the Askey-Wilson polynomials, a four-parameter family of $q$-hypergeometric orthogonal polynomials that, along with the $q$-Racah family, is at the top of $q$-Askey hierarchy. This line of combinatorial inquiry culminated in the work of Corteel and Williams \cite{CorteelWilliams1, CorteelWilliams2}, who, in the context of studying the asymmetric exclusion process (ASEP), gave a combinatorial interpretation for the Askey-Wilson moments. 

Namely, let $Z_\ell$ be the partition function associated with the so-called staircase tableaux of size $\ell$ (that is, $Z_\ell=\sum_T \text{wt}(T)$, where the sum is taken over all such tableaux and the weight is the product of labels assigned to each tableau). Then, the $n$th moment of the Askey-Wilson polynomials in four parameters, $\alpha,\beta,\gamma,\delta$, and one `$q$' is given as \cite{CorteelWilliams1, CorteelWilliams2} 
\begin{equation}
m_n^{(\alpha,\beta,\gamma,\delta,q)}=\frac{(abcd;q)_\infty}{(q, ab, ac, ad, bc, bd, cd ; q)_\infty}\sum_{\ell=0}^n(-1)^{n-\ell}{n\choose \ell}\left(\frac{1-q}{2}\right)^\ell\frac{Z_\ell}{\Pi_{i=0}^{\ell-1}(\alpha\beta-\gamma\delta q^i)}.\label{eq-AW}
\end{equation}
While the formula \eqref{eq-AW} is certainly very interesting, it takes some effort, especially when seeking bijective proofs, to specialize it to the considerably simpler combinatorial interpretations of known special cases of this family. (For this, see \cite{CorteelWilliams2}, which also contains a version of \eqref{eq-AW} without an alternating sum.)

We instead show that a substantial portion of the $q$-Askey scheme can be given a 
transparent combinatorial meaning in terms of the fourteen permutation statistics of Definition~\ref{def-invstats}. Indeed, consider the polynomial sequence $(P_n)$ defined in \eqref{eq-OP}, with the coefficients 
$(\alpha_n)$ and $(\beta_n)$ as in \eqref{eq-alpha-beta}, orthogonal with respect to the family of measures discussed in the preceding section. It is easy to verify that $(P_n)$ includes the sixteen named orthogonal polynomial sequences listed in Table~\ref{table-ops}.  This interpretation straightforwardly recovers previously studied special cases in this scheme, such as \cite{ismail-stanton-viennot-comb-hermite-askey, medicis-stanton-white-charlier, simion-stanton-laguerre,simion-stanton-octabasic,kasraoui-stanton-zeng-salam-chihara-laguerre}, generalizing and unifying these into a consistent combinatorial framework.

The idea that a broader subset of the $q$-Askey scheme may be interpretable in terms of permutation statistics is due to Simion and Stanton (see Section 10 of \cite{simion-stanton-octabasic}). At present, due to a difference in the treatment of fixed points (see Section~\ref{sec-comparisons}), we are able to substantially enlarge this class of examples. It remains to be seen whether this combinatorial framework can be extended to yield a more complete elementary interpretation of the $q$-Askey scheme and related entities. (See, for example, the remark following Corollary~\ref{cor-des-2-31} in Section~\ref{subsec-patterns} concerning the appearance in \cf\ of familiar linearization coefficients.) For now, the further extensions beyond the setting of the symmetric groups, described in Sections~\ref{subsec-hyperocta} and \ref{subsec-arrangements}, and the striking connection to permutation patterns (Section~\ref{subsec-patterns}) provide a new perspective on these well-studied families of orthogonal polynomials.

\subsection{The Hankel transform}\label{subsec-hankel}

In enumerative combinatorics, the passage from a given integer sequence $(a_n)$ to the sequence of determinants of the Hankel matrices $(a_{i+j})_{0\leq i,j \leq n}$  is referred to as the \emph{Hankel transform}. Radoux was among the first to undertake a systematic study of the Hankel transforms of familiar integer sequences \cite{Radoux1979,Radoux1990,Radoux1991,RadouxSLC,Radoux2000,Radoux2002}. For instance, consider the so-called exponential polynomials $(e_n)$, 
\begin{equation}
e_0(x)=1,\quad e_n(x)=\sum_{\pi}x^{|\pi|}\quad (n\in\NN),
\end{equation}
where the sum runs over all set partitions of $[n]$ and $|\pi|$ denotes the number of blocks of a partition~$\pi$. Clearly, $e_n(1)$ is the $n$th Bell number, while the coefficients of $e_n(x)$ are the Stirling numbers of the second kind. In \cite{Radoux1979}, Radoux showed that 
\begin{equation}
\det(e_{i+j}(x))_{0\leq i,j\leq n}=x^{n+1\choose 2} \prod_{k=1}^n k!
\label{eq-radoux}
\end{equation}
(For a combinatorial proof of \eqref{eq-radoux}, see \cite{ehrenborg-hankel}, where it also shown that the identity continues to hold when the exponential polynomials are replaced 
by those enumerating set partitions with block sizes of at least, at most, or exactly equal to $2$.) 
Radoux's result is at the heart of two combinatorial phenomena. First, the Hankel transforms of many of the `classical' sequences, such as the Euler numbers, Catalan numbers, numbers of involutions, and Hermite polynomials, have similarly pleasing forms. Second, many combinatorial sequences share the same Hankel transform. For example, letting $D_n$ denote the number of derangements (that is, permutations without fixed points) on $n$ letters and setting $D_0:=1$, one can show \cite{Radoux1979, ehrenborg-hankel} that
\begin{equation}
\det(D_{i+j})_{0\leq i,j\leq n}=\prod_{i=1}^n (i!)^2\quad\text{for all }n\in\NN.
\end{equation}
In addition, \cite{OEIS} lists ten other known integer sequences (viz. A000023, A000142, A000522, A003701, A010842, A010843, A051295, A052186, A053486, A053487, in addition to the derangements A000166) that share the Hankel transform given above. 
The emergence of a unifying theme was reinforced in subsequent work, under different perspectives (e.g. \cite{Junod2003,Wimp2000}).

It turns out that many of the examples studied in \cite{RadouxSLC,Junod2003,Wimp2000,ehrenborg-hankel} as well as all of the aforementioned sequences in \cite{OEIS} are recovered from straightforward specializations of \eqref{eq-main}. In fact, a considerably more general result is true.  By standard results in the theory of orthogonal polynomials, the determinants of the Hankel matrices associated with the sequence $(m_n)$ are given in terms of products of the diagonal terms of the corresponding continued fraction, namely
\begin{equation}
\det (m_{i+j})_{0\leq i,j\leq n}=(\beta_1)^{n}(\beta_2)^{n-1}\ldots(\beta_{n-1})^2\beta_{n}.
\label{eq-Hankel-product}
\end{equation}
(The reader is referred to~\cite{krattenthaler-adv-det-calc-slc,krattenthaler-adv-det-calc-complement} for historical references and some example applications of~\eqref{eq-Hankel-product}.)  We thereby immediately obtain the following.

\begin{corollary}\label{coro-hankel-det}
For any sequence $(m_n)$ satisfying \eqref{eq-def-main}, we have that
\begin{equation}
\det(m_{i+j})_{0\leq i,j\leq n}= (pr)^{{n+1\choose 2}}\prod_{i=1}^n [i]_{c,d}!\, [i]_{h,\ell}!,
\end{equation}
where $[i]_{c,d}!:=[i]_{c,d}[i-1]_{c,d}\ldots [1]_{c,d}$.
\end{corollary}

In other words, the structure of the Hankel transform as a product of $q$-factorials is broadly shared among integer sequences and generating functions associated with permutations and set partitions. In particular, the above corollary encompasses the combinatorial objects given in Tables~\ref{table-measures} and \ref{table-ops}, as well as further specializations discussed in the next three sections. Overall, the observed similarity among Hankel transforms of familiar sequences rightfully points to a combinatorial phenomenon, the root of which has perhaps mostly to do with the preponderance of combinatorial questions that can be expressed in terms of natural statistics on permutations and set partitions.

\subsection{Permutation Patterns}\label{subsec-patterns}
The continued fraction $\cf$ unifies and refines several major enumerative results on \emph{permutation patterns}.

Concretely, a (classical) permutation pattern of length $\ell$ is a permutation on $\ell$ letters. A permutation $\sigma\in \clsn$ contains an \emph{occurrence of the pattern $\pi$} if it includes $\pi$ as a subpermutation, that is, if there is an $\ell$-tuple 
$i_i<i_2<\cdots<i_\ell$ such that $\sigma(i_j)<\sigma(i_k)$ if and only if $\pi(j)<\pi(k)$.  Moreover,~$\pi$ is a \emph{consecutive pattern} if its occurrences are 
constrained to segments $\sig_k\sig_{k+1}\ldots\sig_{k+\ell-1}$ of $\sig$.  A \emph{vincular pattern} is a pattern partitioned by dashes into blocks whose elements form consecutive patterns, so that in an occurrence of such a pattern, the letters corresponding to a block of consecutive letters must be consecutive in the permutation.  (For example, representing 
elements of $\clsn$ as words on~$[n]$ with each letter appearing exactly once, the subword 624 in 356214 is an occurrence of the vincular pattern $31\dd2$; in contrast, 524 is not since 5 and 2 are not adjacent in 356214, but 524 is an occurrence of the classical pattern $213$). A permutation \emph{avoids} a given pattern if it contains no occurrences of it.

Sporadic, sometimes implicit, results on permutation patterns have appeared for over a century (perhaps first in the work of MacMahon~\cite{macmahon} early last century).  In the 1960s Knuth~\cite{taocp-3} pointed out their connection to sorting devices in theoretical computer science, which has seen much work since then, and in 1985 appeared the seminal paper of Simion and Schmidt~\cite{simion-schmidt} who enumerated the sets of permutations avoiding any given set of patterns of length 3.  In the last two decades this has been an active area of research, with hundreds of papers published, and several new directions arising, involving asymptotics and order theoretic properties and topology of the set of all finite permutations partially ordered by pattern containment.  A fairly recent survey is~\cite{kitaev-book}, which, however, does not cover the more recent 
developments on the order theoretic and topological aspects.

The number of elements of $\clsn$ that avoid any single classical pattern of length 3 is the $n$th Catalan number. The numbers of permutations avoiding one of the vincular patterns $1\dd32$ or $1\dd23$ (and their symmetric equivalents, such as $23\dd1$ and $32\dd1$) are the Bell numbers \cite{cla-gpa}, while avoiders of the vincular pattern $2\dd31$ (which belongs to the only other symmetry class of vincular patterns of length~3 with one dash) are enumerated by the Catalan numbers (which is explicit in~\cite{cla-gpa} 
and implicit in~\cite[Thm. 10]{ClarkeSteingrimssonZeng}, as the statistic $\Res$ there equals the number of occurrences of $2\dd31$).

Both the Catalan and Bell numbers arise as special cases of the continued fraction \cf, which can be seen by comparing \cf\ to the continued fractions given in Proposition~5 and Theorem~2 in~\cite{Flajolet1980}.  (In fact, it is also possible to prove this, and more --- namely, including refinements involving the number of blocks in set partitions and, respectively, the number of descents in permutations counted by the Catalan numbers --- via 
a detailed analysis of the action of the bijection $\phi$ in Lemma~\ref{bij-fond} below.)  Furthermore, the permutations avoiding the consecutive pattern 123 are also enumerated by \cf\ (see Corollary~\ref{cor-des-321} below),
while the numbers of those avoiding the consecutive pattern 132 are not a moment sequence (as evidenced by the computation of a $6\times 6$ Hankel determinant) and, indeed, cannot be recovered from \cf\ (for this would require $\beta_1=1$ and $\beta_5<0$, which is 
inconsistent with~\eqref{eq-alpha-beta}).
These are the only symmetry classes of patterns of length 3. We therefore obtain the 
following corollary  to Theorem~\ref{thm-main}.

\begin{corollary}\label{cor-patterns}
Let $\pi$ be a pattern of length $3$ and denote by $x_n$ the number of elements of $\clsn$ that avoid $\pi$. Then $(x_n)$ is a sequence of moments of some (positive Borel) measure on 
the real line if and only if the ordinary generating function of $(x_n)$ is a special case of $\cf$.
\end{corollary}

While the enumerative results for the patterns encompassed by Corollary~\ref{cor-patterns} are well known, the result itself is more than a sum of its parts. Specifically, the study of permutation patterns varies greatly in difficulty depending on the choice of the pattern, ranging from fine-grained results for consecutive patterns \cite{elizalde-noy} to speculations on ever computing the number of avoiders of the classical pattern $1324$ \cite{DefyingGod}.  The latter is perhaps the most active front in the search for enumeration of single patterns, both exact and asymptotic, and only incremental progress has been made during the last decade in the enumeration of avoiders of this pattern.  Conway, Guttmann and Zinn~\cite{conway-guttmann-zinn-1324} pushed the exact enumeration of 1324-avoiders up to $n=50$ with a powerful algorithm, and used that to give an estimate for the asymptotics of this sequence. Based on this data, Elvey-Price has conjectured~\cite{elvey-price-thesis} that the numbers of 1324-avoiders form a moment sequence. Improving bounds have also been emerging for this pattern \cite{bevan-al-staircase-1324}, but they are still very wide, and no general methods seem to be on the horizon for exact enumeration of $1324$ avoiders, let alone of more intricate patterns.  However, Corollary~\ref{cor-patterns} gives hope that there are nevertheless general threads connecting these seemingly disparate problems and that general results may yet be achievable in this setting.

Corollary~\ref{cor-patterns} does in fact extend to the numbers of occurrences of two of the patterns mentioned above.  In \cite[Thm. 4.1]{elizalde-noy}, Elizalde and Noy gave an explicit (but complicated) expression for the exponential generating function for the number of occurrences of the consecutive pattern $123$ in permutations, and in \cite[Eq.~(5)]{elizalde-cf-fpsac} Elizalde gave a continued fraction for the corresponding ordinary generating function. Of course, this distribution is equal to the distribution of the number of occurrences of the consecutive pattern 321, denoted $\occ_{321}$. As a corollary of our main theorem, the continued fraction $\cf$ refines this result further, giving the joint distribution of $\occ_{321}$ with the number of \emph{descents} (places $i$ such that $\sigma_i>\sigma_{i+1}$). First we recall a bijection that we will use several times in what follows.
\begin{lemma}[{\cite[Thm.~51]{es-indexed}}]\label{bij-fond}
Given $\sig=a_1a_2\ldots a_n\in\clsn$ define $\phi:\clsn\ra\clsn$ by declaring $a_0=0$ and setting $\phi(\sig)=b_1b_2\ldots b_n$, where:
\begin{enumerate}
\item if $a_i>a_j$ for some $j>i$ then $b_{a_{i+1}}=a_i$, that is, $a_i$ is then in place $a_{i+1}$ in $\phi(\sig)$,
\item if $a_i<a_j$ for all $j>i$, find the rightmost letter smaller than $a_i$.  If this letter is $a_k$ then $b_{a_{k+1}}=a_i$, that is, then $a_i$ is in place $a_{k+1}$ in $\phi(\sig)$.
\end{enumerate} 
Then  $\phi$ is a bijection.
\end{lemma}
As an example, $\phi(264135)=413652$. As it turns out, the bijection $\phi$ is a variation on the \emph{transformation fondamentale} of Foata and Sch\"utzenberger~\cite{foata-schutz-eulerien}).  Note that $\phi$ takes  descents `verbatim' to excedances, in the sense that if  $j$ immediately precedes $i$ in  $\sig$ and $j>i$, so that $\ldots ji\ldots$ forms a descent in $\sig$, then $\tau=\phi(\sig)$ has $\tau(i)=j$, so that $i$ is an excedance in $\tau$.

\begin{corollary}\label{cor-des-321}
Setting  ~$s=qx, ~~ p=x$, and all other parameters to 1, we obtain
$$
\cf(z)=\sum_{n\geq 0}\sum_{\sigma\in \clsn} x^{\des(\sigma)}q^{\occ_{321}(\sigma)} z^n,
$$ 
where $\occ_{321}$ is the number of occurrences of the consecutive pattern $321$ and $\des$ denotes the number of descents.
\end{corollary}
\begin{proof}
An occurrence of the pattern 321 is a \emph{double descent}, defined as an $i$ such that 
$\sig_{i-1}>\sig_{i}>\sig_{i+1}$.  Suppose a permutation $\sig$ has a double descent involving the letters $i,j,k$, appearing in that order, so $i>j>k$. Applying $\phi$ to $\sig$ will result in a permutation $\tau$ with $\tau(k)=j$ and  $\tau(j)=i$, which is precisely a linked excedance.  Since the parameter~$s$ carries linked excedances in \cf, each linked excedance contributes a factor of both~$x$ and $q$ here, whereas the non-linked excedances, carried by $p$ in \cf, contribute only an~$x$. 
\end{proof}

This connection between the occurrences of the pattern $321$ and the continued fraction \cf\ is fundamental. Indeed, each occurrence of the consecutive pattern $321$ corresponds, in the simple manner described above, to a linked excedance, which is one of the basic 
statistics carried by $\cf$. In an analogous manner, descents correspond to arbitrary 
excedances.\footnote{The double descents appearing in the proof of Corollary~\ref{cor-des-321} are related to the `double falls', whose distribution in the form of a continued fraction was given by Flajolet~\cite[Theorem 3A]{Flajolet1980}.  These are defined as our double descents, except that Flajolet's double falls count also a single descent at the end of a permutation, since he appends a 0 to a permutation when computing these.}

It is further worth noting that the bijection $\phi$, in addition to pairing linked excedances with double descents, pairs linked anti-excedances with double ascents.  Moreover, it is straightforward to show that if $\sig$ has no linked excedances, no linked anti-excedances and no fixed points, then $\phi(\sig)$ has no double descents or double ascents and also begins and ends with a descent (and thus must have even length).  In other words, $\phi(\sig)$ is then an \emph{alternating permutation} and these are known to be counted by the (even) Euler numbers when $n$ is even.  This bijection thus proves the following, since $s$, $t$ and $u$ carry, respectively, linked excedances, linked anti-excedances and fixed points:
\begin{corollary}\label{coro-alt-perms}
  The number of permutations in $\clsn$ with no fixed points, no linked excedances and no linked anti-excedances is the Euler number $E_n$ if $n$ is even, and 0 if $n$ is odd. The generating function for the even Euler numbers is therefore obtained from $\cf$ by setting $s=t=u=0$ and all other parameters to 1.
\end{corollary}

\begin{remark}
  If we also set $h=0$ in Corollary~\ref{coro-alt-perms} the corresponding alternating permutations correspond exactly to perfect matchings.  Namely, $h$ carries inversions among non-linked anti-excedances (and thus among all anti-excedances when there are no linked anti-excedances).  But the anti-excedances of $\sig$ are the `descent bottoms' of $\phi(\sig)$, which are the letters in even-numbered places in~$\sig$.  An example is the alternating permutation $71428365$, and this is the unique such permutation corresponding to the perfect matching $(1,7)-(2,4)-(3,8)-(5,6)$.
\end{remark}

The continued fraction $\cf$ also encodes the number of occurrences of the vincular pattern $2\dd31$.  This distribution was already determined by Claesson and Mansour~\cite[Thm. 22]{cla-mans-count} who gave a continued fraction for the joint distribution with the number of descents, the number of ascents and of the pattern $31\dd2$.  At present, we recover from \cf\ the following bivariate generating function, the proof following by a 
simple manipulation of the continued fraction in Theorem 10 in \cite{ClarkeSteingrimssonZeng}, noting that the statistic $\Res$ in that paper equals number of occurrences of $2\dd31$.
 
\begin{corollary}\label{cor-des-2-31}
With ~$b=d=g=\ell=q,~~ s=xq,~~ p=u=x$, and all other parameters set to~1, we have
$$
\cf(z)=\sum_{n\geq 0}\sum_{\sig\in \clsn} x^{\des(\sig)+1} q^{\occ_{2\sdd31}(\sig)}  z^n.
$$ 
\end{corollary}

There is a more illuminating way of proving Corollary~\ref{cor-des-2-31}, which shows how 
this distribution arises naturally from $\cf$.  Namely, it was shown in ~\cite[Thm.~24]{stein-williams} that the distribution of $(\des+1,\occ_{2\dd31})$ equals that of $(\wex,\cross)$ on $\clsn$, where $\wex(\sig)$ is the number of \emph{weak excedances} in $\sig$, that is, $i$ such that $\sig(i)\ge i$, which is equal to $\exc(\sig)+\fp(\sig)$.  The number of \emph{crossings} $\cross(\sig)$ was defined in~\cite[Def.~1]{Corteel2007} (see~\cite[Def.~13]{stein-williams}), and a straightforward comparison shows that it equals $\nie(\sig)+\niae(\sig)+\lexc(\sig)$, leading to the parameter assignments in the corollary.  It is interesting to note that this distribution was in~\cite[Thm.~5]{kasraoui-stanton-zeng-salam-chihara-laguerre} shown to coincide with 
the \emph{linearization coefficients of the $q$-Laguerre polynomials}.

Overall, the continued fraction $\cf$ encodes a variety of combinatorial objects associated to permutation patterns and permutation statistics. As such, it allows us to recover and refine several key enumerative results in this area, while providing a unifying lens through which these can be viewed. In particular, the fact that \cf\ encodes all avoiders of patterns of length~3 (classical, vincular and consecutive) that are moment sequences gives us hope that the seemingly specialized questions in permutation patterns may be more effectively approached as instances of a more general problem.

\subsection{Generalization to the signed and colored permutations}\label{subsec-hyperocta}

The continued fraction $\cf$ enumerates elements of the symmetric groups according to 
the fourteen combinatorial statistics described in Definition~\ref{def-invstats}. By rather simple substitutions of parameters, however, we obtain the distribution of analogous statistics on \emph{$k$-colored permutations}, which are permutations on $n$ letters, where each letter is assigned an integer (color) from $\{0,1,\ldots,k-1\}$. The case $k=2$ specializes to the hyperoctahedral group of signed permutations.  More generally, the colors of the letters factor into the definitions of descents and various other classical permutation statistics, as first defined in~\cite{es-indexed} for general $k$ and outlined at present.

\begin{definition}\label{def-colored-perms}  
  The set of \emph{$k$-colored} permutations of $[n]$, denoted $\snk$, consists of all pairs $(\sig,c)$ where $\sig=a_1a_2\ldots a_n$ is  a permutation of $[n]$ and $c=c_1c_2\ldots c_n$ where each $c_i$ is an element of $\{0,1,\ldots,k-1\}$.  We refer to $c_i$ as the color of the letter $a_i$.  We define a total order on the pairs $(a_i,c_i)$ by setting  $(a_i,c_i)>(a_j,c_j)$ if $c_i>c_j$, or $c_i=c_j$ and $a_i>a_j$.  Moreover, when needed for definitions of statistics, we declare $a_{n+1}=n+1$ and $c_{n+1}=0$.
\end{definition}

We now define five statistics on $\snk$, following~\cite{es-indexed}:
\begin{definition}\label{def-colperm-stats}
  Let $P=(a_1a_2\ldots a_n, c_1c_2\ldots c_n)$ be a $k$-colored permutation in $\snk$:
  \begin{enumerate}
  \item A \emph{descent} in $P$ is an $i\le n$ such that $(a_i, c_i)>(a_{i+1},c_{i+1})$. The number of descents in $P$ is denoted $\des(P)$.
  \item An \emph{excedance} in $P$ is an $i\le n$ such that $a_i>i$, or $a_i=i$ and $c_i>0$. The number of excedances in $P$ is denoted $\exce(P)$.
  \item An \emph{anti-excedance} in $P$ is an $i\le n$ such that $a_i<i$. The number of anti-excedances in $P$ is denoted $\aexc(P)$.
  \item An \emph{inversion} in $P$ is a pair $(i,j)$ such that $i<j\le n+1$ and $(a_i, c_i)>(a_j,c_j)$.  The number of inversions in $P$ is denoted $\inv(P)$.
  \item A \emph{fixed point} in $P$ is an $i\le n$ such that $a_i=i$ and $c_i=0$. The number of fixed points in $P$ is denoted $\fix(P)$.
  \end{enumerate}
\end{definition}

Note that in contrast to the symmetric group ($k=1$), $n$ is a descent in $\sig=(\pi,c)\in\snk$ when $a_n$ has a positive color ($c_n>0$), and $n$ is an excedance when $a_n=n$ and $c_n>0$. Anti-excedances, however, coincide with anti-excedances of $\pi$. Also, any $i\le n$ forms an inversion with $(n+1)$ if and only if $c_i>0$. Moreover, a fixed point requires $a_i=i$ and $c_i=0$, which matches the definition of excedances including fixed points $i$ of $\pi$ with $c_i>0$.

The best known case of colored permutations is that of the signed permutations ($k=2$), frequently viewed as elements of the hyperoctahedral groups. The signed permutations have been extensively studied from many different perspectives, and most statistics on the symmetric groups have been generalized to these and analogous results obtained about the distributions of such statistics.

The colored permutations can be seen as elements of certain unitary groups generated by 
reflections (see~\cite{shephard-unitary-groups}) that coincide with the wreath product of the symmetric group $\clsn$ with the cyclic group~$\ZZ_k$.  Some of the classical permutation statistics were generalized to these colored permutations and many of their properties elicited in \cite{es-indexed}, and in several subsequent papers by other authors (see~\cite{bagno-garber-mansour-color-perms} for some references).

We start by showing that $\cf$ encodes the Eulerian numbers of type B, that is, the descent statistic on the signed permutations, prior to generalizing the result to all colored permutations.  We do this for excedances --- which were shown in~\cite{es-indexed} to be equidistributed with descents --- as part of a joint distribution with anti-excedances and fixed points.

In~\cite[Thm. 4.2]{foata-schutz-eulerien}, Foata and Sch\"utzenberger gave an exponential generating function for the joint distribution of $(\exc,\fix)$ on the symmetric group.  We now give this joint distribution over the $k$-colored permutations, incorporating also anti-excedances.  
\begin{corollary}\label{thm-exc-fix}
  By setting $s=p=kx,~t=r=ky,~u=(k-1)x+q$ in \cf, and all other parameters to~1, 
\cf\ becomes the generating function of the joint distribution of excedances and fixed points on the colored permutations $\snk$, that is, 
$$
\cf(z) ~=~ \sum_{n\geq 0}\sum_{\sig\in\snk}x^{\exc(\sig)}y^{\aexc(\sig)} q^{\fix(\sig)} z^n.
$$
\end{corollary}

\begin{proof}
  Setting $s=p=kx$ lets $x$ carry linked excedances, since, letting $\sig=(\pi,c)$, an excedance~$i$ in $\pi$ is an excedance in $\sig$ regardless of $c_i$, the color.  The same is true for non-linked excedances carried by $p$, and likewise for anti-excedances, carried by $t$ and $r$.  The only complication is that a fixed point $i$ in $\pi$ is a fixed point of $\sig$ only if $c_i=0$, and is otherwise an excedance.  This is reflected in~$u$, which carries fixed points in the case of the symmetric group $(k=1)$. Since $u$ is now the sum of~$q$ and $(k-1)x$, that is, $q$  for the fixed point case when $c_i=0$ and  $(k-1)x$ for the remaining $k-1$ choices for~$c_i$, we recover excedances at $i$.
\end{proof}
Although we present Corollary~\ref{thm-exc-fix} only in this simple form, it is 
straightforward to refine it to incorporate the obvious generalizations of linked vs. non-linked excedances and anti-excedances to colored permutations.

Setting $w=0$ in \cf\ gives $u\cd w^n=0$ for all $n\in\NN$, while $u\cd w^0=u$.  This allows us to obtain the distribution of the number of inversions on the symmetric group and, more generally, on the $k$-colored permutations, as follows:
\begin{corollary}\label{thm-inv}
  With 
$$
a=c=h=r=q, ~~~~ b=f=d=\ell=t=q^2, ~~~~ g=w=0, ~~~~ p=u=1, ~~~~ s=2q,
$$
the continued fraction $\cf$ recovers the generating function for the distribution of inversions over $\clsn$:
$$
\cf(z) ~=~ \sum_{n\geq 0}\sum_{\pi\in\clsn}q^{\inv(\pi)}z^n.
$$
If, in addition, we replace $z$ by $z((k-1)q+1)$, then $\cf$ gives the distribution of inversions over the $k$-colored permutations $\snk$ for $k>1$.
\end{corollary}
\begin{proof}
In \cite[Equation 1.1]{biane-exc-inv-heine}, Biane gives a continued fraction for the distribution of the number of inversions on $\clsn$.  It is 
a straightforward calculation to show that Biane's continued fraction coincides with our $\cf$ with parameters set as in the statement of this proposition.

Replacing $z$ by $z\cdot((k-1)q+1)$ is equivalent to attaching to each letter in a permutation any one of $(k-1)$ different $q$s (which we don't need to distinguish between) 
or nothing, the latter corresponding to the unadorned $z$ in the case of the symmetric 
group $(k=1)$. With $(\pi,c)=\sig\in\snk$, this corresponds to the fact that only one choice of $c_i$ (namely 0) for a letter $a_i=\pi(i)$ with $i\le n$ contributes a non-inversion with the appended letter $a_{n+1}=n+1$ (which has color 0). Each of the other $(k-1)$ choices for $c_i$ contributes 1 to the total number of inversions for that permutation.
\end{proof}

Moreover, in the symmetric group, the above can be further refined to a joint distribution of inversions and excedances:

\begin{corollary}\label{thm-exc-inv}
Setting $p=x$, $s=(1+x)q$, and all other parameters as in Corollary~\ref{thm-inv}, yields
$$
\cf(z) ~=~ \sum_{n\geq 0}\sum_{\pi\in\clsn}x^{\exce(\pi)}q^{\inv\pi}z^n.
$$
\end{corollary}
\begin{proof}
  In \cite[Thm. 10]{ClarkeSteingrimssonZeng} a continued fraction is given that captures the distribution of $(\exce,\inv)$ on the symmetric group, since $\inv$ is the sum of the statistics $\Edif$ and $\Ine$ in that paper.  This continued fraction, expressed in the form for our $\cf$, and setting $p=q$ to let $q$ carry $\inv$, has $\widetilde\alpha_n$ and $\widetilde\beta_n$ as follows:
  \begin{equation*}
\widetilde\alpha_n=q^n(x\cdot[n]_{q}+[n+1]_{q}),\quad \quad
\widetilde\beta_n=x\cdot q^{2n-1}([n]_{q})^2.
\end{equation*}
The result follows by a straightforward comparison to \eqref{eq-alpha-beta} with the above parameter choices.
\end{proof}

Whether the joint distribution of excedances and inversions over $\snk$ can be obtained 
from $\cf$ with appropriate parameter settings remains to be seen. Also worth investigating is whether other Euler-Mahonian pairs (see e.g.~\cite{shareshian-wachs-exc-maj} for examples) can be obtained from $\cf$, such as descents and the \emph{major index}, which is the sum of the descents in a permutation.

\subsection{The $k$-arrangements}\label{subsec-arrangements}

The fact that the parameter $u$ in $\cf$ carries fixed points in permutations suggests an obvious generalization.  We define a \emph{$k$-arrangement of $[n]$}, for any $k\in\NN_0$, to be a permutation of $[n]$ each of whose fixed points is colored with any of $k$ colors, with the convention that a 0-arrangement is a \emph{derangement} of $[n]$ (a permutation with no fixed points).  Letting $k=1$ recovers ordinary permutations
and $k=2$ corresponds to what previously have been called simply `arrangements' \cite{comtet-advanced}, which also coincide with Postnikov's definition of `decorated permutations'~\cite[Def. 13.3]{postnikov-webs}.  

Equivalently, a $k$-arrangement of $[n]$ is a choice of $k$ disjoint subsets (some possibly empty) of $[n]$ together with a derangement of the complement of their union.  An example of a 3-arrangement of $[7]$, where we denote the three available colors by $a,b,c$, is 6214573, with the fixed points 2,4,5 colored $c,a,a,$ respectively.  This could be presented as $6c1aa73$, or, more concisely, as $3c1aa42$, to present the derangement part as a derangement of $\{1,2,3,4\}$. (Note that the original values of the letters in the derangement part can
be recovered from their relative order and the set of fixed points, so the two are indeed equivalent.)  Formally:

\begin{definition}\label{def-arrangements}
Let $n\in\NN$. 
For a permutation $\pi\in\clsn$, denote by $\fixpts(\pi)$ the set of fixed points of $\pi$.
  \begin{enumerate}
  \item For any $k\in\NN$, a \emph{\karr\ of $[n]$} is a pair $\a=(\pi,\phi)$ where $\pi$ is a permutation of $[n]$  and $\phi:\fixpts(\pi)\ra[k]$ an arbitrary map.  A \emph{0-arrangement} of~$[n]$ is a derangement of~$[n]$.

  \item For any $k\in\NN_0$, the set of \karrs\ of $[n]$ is denoted $\ank$.  The set $\arr{k}{0}$, for any $k\ge0$, consists of the empty word.
  \end{enumerate}
\end{definition}

In what follows, when we apply the statistics on $\clsn$ in Definition~\ref{def-invstats} to $\a=(\pi,\phi)\in\ank$ we are applying them to $\pi$.  This yields the following generalization of Theorem~\ref{thm-main}.

\newcommand\col{\mathrm{col}}

\newcommand\bu{\ensuremath{\mathbf{u}}}
\begin{proposition}\label{prop-gf-arr}
  Let $\bu=u_1+u_2+\cdots+u_k$.  
Given $\a=(\pi,\phi)\in\ank$ let 
\begin{equation}\col(\a)=\prod_{i\in\fixpts(\pi)}u_{\phi(i)}.\end{equation}
Then, letting $\cfu$ be the continued fraction $\cf$ in Definition~\ref{def-gz} with $\bu$ substituted for $u$, we have
\begin{equation}\label{eq-gf-arr}
  \begin{array}{l l l}
\cfu(z)=&\displaystyle\sum_{n\geq 0}\sum_{\a\in\ank}&
\col(\a)\,\csummand
  \end{array}
\end{equation}
\end{proposition}
\begin{proof}
Note first that all the statistics in the exponents of parameters in~\eqref{eq-gf-arr} carry over directly to \karrs\ from Theorem~\ref{thm-main}, since they are applied only to the permutation $\pi$.

\newcommand\ms{\ensuremath{\mathsf{S}}}

Let $\mathsf{S}$ denote the summand in \eqref{eq-comb-main} (equivalently, the summand in~\eqref{eq-gf-arr} take away
$\col(\a)$). Replacing $u$ in $\cf$ by $\bu$ means replacing
$u^{\fp(\sig)}\ms$ by $\bu^{\fp(\sig)}\ms$ in $\cf$.  That is equivalent to
attaching $\bu^{\fp(\sig)}$ to each permutation $\sig\in\clsn$ (by Theorem~\ref{thm-main} and the underlying bijection described in Section~\ref{section-bijection}).  Expanding
$\bu^{\fp(\sig)}$ non-commutatively turns it into $k^{\fp(\sig)}$ monomials
in the $u_i$, each of degree $\fp(\sig)$.  Each such monomial corresponds in
a straightforward way to a coloring of the fixed points of $\sig$, where the
$i$th fixed point gets color $u_j$ if $u_j$ is the $i$th factor in the
corresponding monomial. Conversely, each coloring of fixed points in $\sig$
corresponds in the same simple way to one of the monomials in
$\bu^{\fp(\sig)}$, where the $i$th factor in such a monomial is the color $u_i$ of the $i$th fixed point.  But the colorings of fixed points in $\sig$ are represented precisely by the monomials $\col(\a)$.  Thus, replacing $u^{\fp(\sig)}$ by $\bu^{\fp(\sig)}$ in $\cf$, and hence in its expansion as in Theorem~\ref{thm-main}, corresponds to summing over all $\a\in\ank$ with a factor of $\col(\a)$ for each of them replacing $u^{\fp(\sig)}$.
\end{proof}

\begin{remark}
Since replacing  $u$ by $u_1+u_2+\cdots+u_k$ in \eqref{eq-alpha-beta} only affects the terms $\alpha_n$ and not $\beta_n$, the results in Sections~\ref{subsec-permutation-moment} and~\ref{subsec-hankel}, including Corollaries~\ref{cor-moments} and~\ref{coro-hankel-det}, are satisfied by $\cfu$.
\end{remark}
We now define two representations of \karrs, which give more insight into their structure. 
\begin{definition}\label{def-der-form}
  The \emph{derangement form} of an element of $\ank$ is a word $\a$ on the alphabet
$$
\{-k,-k+1,\ldots,-1,1,2,\ldots,n-K\},
$$
where $K$ is the number of negative letters in $\a$, the positive letters appear once each, and the restriction of $\a$ to its positive letters is a derangement of $[n-K]$, called the \emph{derangement part of~$\a$}.\\  
The \emph{permutation form} of an element of $\ank$ is a word $\a$ on the alphabet
$$
\{-k+1,\ldots,-1,1,2,\ldots,n-K\},
$$
where $K$ is the number of negative letters in $\a$, the positive letters appear once each, and the restriction of $\a$ to its positive letters is a permutation of $[n-K]$, called the \emph{permutation part of~$\a$}.
\end{definition}

Given a \karr\ $\a\in\ank$, expressed as a permutation $\pi=a_1a_2\ldots a_n$ together with a map $\phi:\fixpts(\pi)\ra[k]$ we obtain the derangement form of $\a$ by changing $a_i$ to $-\phi(i)$ for each fixed point~$i$ of~$\pi$ and then replacing the $j$-th largest positive letter left by $j$.  To get the permutation form of~$\a$ we do the same except that we leave fixed points $i$ with $\phi(i)=k$ intact before modifying the positive letters.  Each of these transformations is clearly invertible, justifying Definition~\ref{def-der-form}.

As an example, the 4-arrangement $\pi=6214573$ with $\phi(2)=4,\phi(4)=1,\phi(5)=4$ has derangement form 3\dd41\dd1\dd442, whose derangement part is 3142, and permutation form 521\dd1463, whose permutation part is 521463.  As a \karr\ for any $k>4$, however, the permutation form of $(\pi,\phi)$ equals the derangement form.

The numbers $A_k(n)$ of \karrs\ of $[n]$ satisfy a simple recursion equivalent to a simple exponential generating function, showing them to be an infinite family enumerated by successive binomial transforms:

\begin{proposition}\label{prop-arrangement-formula}
Let $A_k(n)$ be the number of \karrs\ of $[n]$. Then
\begin{enumerate}
\item $A_k(0)=1$ and, for $n>0$, $A_k(n)=n\cd A_k(n-1)+(k-1)^n$.
\item The exponential generating function for $A_k(n)$ is ~ $\sum_{n\ge0}A_k(n)\frac{x^n}{n!}=e^{(k-1)x}/(1-x)$.
\item $A_k(n)$ equals the permanent of the $n\times n$ matrix with $k$ on the diagonal and 1s elsewhere.
\item The sequence $\left(A_k(n)\right)_{n\ge0}$ is the binomial transform of the sequence $\left(A_{k-1}(n)\right)_{n\ge0}$, that is, $A_k(n)=\sum_{i\ge0}\binom{n}{i}A_{k-1}(i)$.
\end{enumerate}
\end{proposition}
\begin{proof}
We will show how to generate all elements of $\ank$ from those of $\arr{k}{n-1}$, counting them along the way.
  Recall from Definition~\ref{def-arrangements} that the permutation form of an element of $\arr{k}{n-1}$ is a word $\a$ consisting of any $K$ negative letters from $\{-k+1,-k+2,\ldots,-1\}$, where $0\le K\le n-1$ and the positive letters $\{1,2,\ldots,n-1-K\}$ once each.
  For any such \karr\ $\a$ let $M$ be its maximum letter, except $M=0$ if all letters of $\a$ are negative. We can then insert a new maximum $M'=M+1$ in any one of the $n$ places before or after a letter in $\a$ to obtain words with letters from $\{-k+1,-k+2,\ldots,-1,1,2,\ldots,n-K\}$ that contain at least one positive letter.  Clearly this is injective, since if two resulting words are equal they have the unique maximum letter, in the same place, and removing it gives two identical preimages.  Conversely, any \karr\ of $[n]$ that has some positive letters has a unique maximum letter that when removed leaves a \karr\ of $[n-1]$ (possibly with only negative letters).

This thus produces, from the \karrs\ of $[n-1]$, all \karrs\ of $[n]$ that contain at least one positive letter, and thus $n$ such \karrs\ of $[n]$ for each \karr\ of $[n-1]$, or $n\cd A_k(n-1)$ in total.  What remains are the \karrs\ of $[n]$ with only negative letters, of which there are $(k-1)^n$.

This recursion leads to the exponential generating function claimed, via the known exponential generating function for derangements ($k=0$), induction and a straightforward manipulation.

That $A_k(n)$ equals the permanent of the $n\times n$ matrix  $M_n^k$ with $k$ on the diagonal and 1s elsewhere is easily derived from the definition of the permanent and the fact that the permanent of $M_n^1$ counts permutations of $[n]$.  Namely, a $k$ appearing in the expansion of the permanent of $M_n^k$ corresponds to a fixed point in a permutation, as each $k$ comes from the diagonal of $M_n^k$, and thus counts the $k$ possible choices for coloring that fixed point.

Finally, it is well known (and easy to prove) that the binomial transform of an exponential generating function $f(x)$ is $e^x\cd f(x)$ which implies $\left(A_k(n)\right)_{n\ge0}$ is the binomial transform of $\left(A_{k-1}(n)\right)_{n\ge0}$.
\end{proof}

\begin{remark}
Interestingly, all the properties of Proposition~\ref{prop-arrangement-formula} hold even for negative~$k$, and it seems that for any~$k$ the $A_k(n)$ eventually become positive for 
$n$ sufficiently large, so an obvious question is whether the $A_k(n)$ also enumerate some interesting combinatorial structures for $k<0$.
\end{remark}

The definitions of descents, excedances, inversions and major index for permutations 
generalize in a natural way to the derangement and permutation forms of \karrs\ as follows.

\begin{definition}\label{def-arr-stats}
  Given the derangement form or permutation form $\a=a_1a_2\ldots a_n$ of a \karr\  let $s_1s_2\ldots s_n$ be the word consisting of the letters of $\a$ sorted in weakly increasing order.
  \begin{enumerate}
  \item An \emph{excedance} in $\a$ is an $i$ such that $a_i>s_i$. The number of excedances in $\a$ is denoted $\exce(\a)$.
  \item A \emph{descent} in $\a$ is an $i$ such that $a_i>a_{i+1}$. The number of descents in $\a$ is denoted $\des(\a)$.
  \item The \emph{major index} of $\a$, denoted $\maj(\a)$, is the sum of the descents of $\a$.
  \item An \emph{inversion} in $\a$ is a pair $(i,j)$ such that $i<j$ and $a_i>a_j$. The number of inversions in $\a$ is denoted $\inv(\a)$.
  \end{enumerate}
\end{definition}

We present below some results and several conjectures for the above statistics.  Which other statistics on permutations generalize in natural ways to \karrs\ remains to be elicited, as well as the distributions of such statistics.  First we present a property of 
\karrs\ in the permutation form, allowing us to exploit the substantial knowledge on rearrangement classes, which are sets of permutations of multisets.

\begin{definition}
  Let $\a$ be a word of length $n$ on some alphabet, and $M(\a)$ the multiset of all letters in $\a$.  The \emph{rearrangement class of $\a$} is the set of words of length $n$ whose multiset of letters is $M(\a)$.
\end{definition}

As an example, the rearrangement class of 0110 consists of $\{0011, 0101, 0110, 1001, 1010, 1100\}$.

\begin{proposition}\label{prop-karr-rearr}
  The set of elements of $\ank$ in permutation form, for any $n$ and $k$, is a union of rearrangement classes.
\end{proposition}
\begin{proof}
  The permutation form of a \karr\ $\a\in\ank$ consists of any multiset of size $s\le n$ of negative numbers in $\{-1,-2,\ldots,-k+1\}$ and each of the positive numbers in $[n-s]$, once each.  Every rearrangement of $\a$ consists of the same multiset of negative numbers and the same positive numbers and thus is the permutation form of some element of $\ank$.
\end{proof}

It was shown already by MacMahon~\cite{macmahon} early last century that $\exc$ and $\des$ (as defined above) are equidistributed on the rearrangement class of any word, and that the same applies to $\inv$ and $\maj$.  Although rearrangement classes are typically defined on words containing all the letters in $[k]$ for some $k$, statistics on them often depend only on the relative sizes of letters.  Thus, thanks to Proposition~\ref{prop-karr-rearr}, we have the following result. 

\begin{proposition}
Suppose $s$ and $s'$ are two statistics on words of integers that depend only on the relative sizes of those integers, and suppose $s$ and $s'$ are equidistributed on the rearrangement class of any word.  Then $s$ and $s'$ are equidistributed on $\ank$ in permutation form for any $n$ and $k$. 
\end{proposition}

Since the definitions of $\exc$, $\des$, $\inv$ and $\maj$ only depend on relative sizes of letters in \karrs, we have the following result.

\begin{proposition}
  The statistics $\exc$ and $\des$, and, respectively, $\inv$ and $\maj$, are equidistributed on elements of $\ank$ in their permutation form for any $n$ and $k$.
\end{proposition}

\begin{conjecture}\hskip-.4em\footnote{In response to the preprint of this article posted on {\tt arxiv.org}, Fu, Han and Lin~\cite{fu-han-lin-k-arrangements} have recently proved Conjectures 1 through 4 in this section.} \hspace{-.5ex}
  The statistic $\des$ has the same distribution on elements of $\ank$ in their 
derangement form and permutation form for any~$n$ and~$k$.
\end{conjecture}

In the remainder of this subsection, all permutation patterns are taken to be classical.

It is well known that the number of permutations avoiding any single permutation pattern of length 3 is the $n$-th Catalan number $C(n)=\frac{1}{n+1}\binom{2n}{n}$. An analogous statement is true of the 2-arrangements in their permutation form.  The definition of occurrence of a pattern extends trivially to \karrs\ in their permutation (or derangement) form, but an occurrence of such a pattern must involve different letters, and can thus not have more than one copy of any of the negative letters. (There are, however, generalizations of patterns that allow repeated letters in occurrences, and studying these for the \karrs\ would be natural. For information on such patterns see~\cite{kitaev-intro-partord-patts} and \cite[Section 1.5]{kitaev-book}).

\begin{proposition}\label{conj-arr-patt}
  The number of 2-arrangements of $[n]$ whose permutation form avoids any single 
permutation pattern of length 3 is the Catalan number $C(n+1)$.  Moreover, the number of such 2-arrangements with $k$ negative numbers in their permutation form is the ballot number $\frac{k+1}{n+1}\binom{2n-k}{n}$.
\end{proposition}
\begin{proof}
We first exhibit a bijection between 312-avoiding 2-arrangements of $[n-1]$ and Dyck paths 
of semilength $n$. Such a Dyck path consists of upsteps $(1,1)$ and downsteps $(1,-1)$, starting from $(0,0)$, ending at $(2n,0)$ and never going below the $x$-axis.  A Dyck path can be decomposed into shorter such paths, according to where it touches the $x$-axis.  The part strictly between two consecutive such steps, the first an upstep from the axis, the latter a downstep to the axis, can be any Dyck path of the appropriate length, lifted one unit.

In a 312-avoiding permutation form of a 2-arrangement the positive letters to the right of any $-1$ must each be larger than all the positive letters preceding that $-1$.  The $-1$s thus partition the positive letters into blocks (possibly empty), each block consisting of letters determined by its size and the number of positive letters to its right, and each block can contain an arbitrary 312-avoiding permutation of those letters.  The 312-avoiding permutations of $[m]$ are well known to be in bijection with Dyck paths of semilength $m$ and so this structural decomposition is identical to that of Dyck paths described above. This correspondence pairs Dyck paths of semilength $n$ with arrangements of $[n-1]$ because we must prepend a $-1$ to ensure that the first block of positive letters is preceded by a $-1$.

It follows that 2-arrangements of $[n-1]$ whose permutation form avoids 312 are in bijection with 312-avoiding permutations of $[n]$.  It was shown in~\cite[Thm. 3]{savage-wilf-patts-comp-multisets} that the number of permutations of a given multiset that avoid a pattern of length 3 is independent of the pattern.  As we showed in Proposition~\ref{prop-karr-rearr}, the set of \karrs\ for any $n$ and $k$ is a union of rearrangement classes, and the permutations of a multiset are precisely the rearrangement class of any word consisting of the letters of that multiset.

Finally, since the $-1$s in a 312-avoiding 2-arrangement of $[n-1]$ correspond bijectively to all but the last return to the $x$-axis of the corresponding Dyck path, the distribution of the number of $-1$s claimed follows from a well known fact about Dyck paths.
\end{proof}

\begin{conjecture}
  The number of 3-arrangements of $[n]$ whose permutation form avoids any single pattern of length 3 is $C(n+2)-2^n$.
\end{conjecture}

\begin{conjecture}
The distribution of the number of descents on 2-arrangements of $[n-1]$ whose permutation form avoids any single one of the patterns 132, 213, 231 or 312, is given by the triangle sequence {\tt A108838} in~\cite{OEIS}, which counts, among other things, rooted ordered trees with $n$ non-root nodes and $k$ leaves, and has formula $\frac{2}{(n+1)}\binom{n+1}{k+2}\binom{n-2}{k}$.
\end{conjecture}

\begin{conjecture}
The distribution of the number of ascents on 2-arrangements of $[n-1]$ whose permutation form avoids the pattern 123, is given by triangle sequence {\tt A236406} in~\cite{OEIS}, which counts 123-avoiding permutations of $[n]$ with $k$ peaks.
\end{conjecture}

Given a \karr\ $\a$ as a permutation $\pi$ and a map $\phi:\fixpts(\pi)\rightarrow[k]$, define the \emph{color-encoding} of $\a$ to be the $k$-colored permutation $(\pi,c)$ for which $c_i=\phi(i)$ if $i\in\fixpts(\pi)$ and $c_i=0$ otherwise.  We can then apply the statistics in Definition~\ref{def-colperm-stats} to the color-encoding of a \karr.

\begin{conjecture}
  The statistics $\inv$ and $\maj$ in Definition~\ref{def-colperm-stats} are equidistributed on $\ank$ in the color encoding for any $n$ and~$k$.  The statistic $\des$ in that definition has the same distribution on $\ank$ as $\des$ on the permutation or derangement form for any $n$ and $k$.
\end{conjecture}

The following also seems worthy of investigation.
\begin{problem}
  Proposition~\ref{prop-gf-arr} shows how the continued fraction $\cfu$ and its ordinary generating function capture our fundamental statistics for permutations together with the coloring of fixed points in \karrs.  Is there a specialization of $\cfu$, or a similar continued fraction, that captures the distribution of the statistics on colored 
permutations in Definition~\ref{def-colperm-stats} applied to \karrs?  What about the statistics defined for the derangement and permutation forms in Definition~\ref{def-arr-stats}?
\end{problem}

\section{The Central Bijection}\label{section-bijection}

In this section we describe the bijective map $\eta:\clsn\to\motzn$, whose existence was asserted in Theorem~\ref{thm-bij}. The map $\eta$ takes permutations, refined in terms of fourteen combinatorial statistics introduced in Definition~\ref{def-invstats}, to Motzkin paths with labels as prescribed in Definition~\ref{def-paths}.  

\subsection{Preliminaries}

We start by observing that when the parameters $a$ through $w$ are set to 1, each Motzkin 
path labeled as in Definition~\ref{def-paths} carries unit weight. Since the same parameter setting in \eqref{eq-alpha-beta}
gives $\alpha_n=2n+1$ and $\beta_n=n^2$,  it then follows easily (see e.g.\ \cite[Theorem 3B]{Flajolet1980}) that $\cf$ becomes the continued fraction expansion of the ordinary generating function for the sequence $(n!)$. In other words:
\begin{lemma}\label{lemma-num-motz}
  The number of Motzkin paths $\motzn$ labeled as in Definition~\ref{def-paths} is $n!$.
\end{lemma}

To define the map $\eta$, it is convenient to introduce `vectorized' versions of the permutation statistics of Definition~\ref{def-invstats}. 

\newcommand\sigx{\sigma(x)}
\newcommand\sigi{\sigma(i)}
\newcommand\sigj{\sigma(j)}

\begin{definition}\label{vector-stats}
  Given $\sig\in\clsn$ and $i\in[n]$, let
\begin{itemize}
\item $\inve_i(\sig)=\#\{x \in [n] \mid x < i < \sig(i) < \sig(x)\}$.
\item $\ninve_i(\sig)=\#\{x \in [n] \mid x < i < \sig(x) < \sig(i)\}$.
\item $\inva_i(\sig):=\#\{x \in [n] \mid x > i> \sig(i) > \sig(x)\}$.
\item $\ninva_i(\sig):=\#\{x \in [n] \mid x > i > \sig(x) > \sig(i)\}$.
\item $\iefp_i(\sig):=\#\{x \in [n] \mid x < i < \sig(x)\}$.
\end{itemize}
\end{definition}
Here the conditions on $x,i,\sigx,\sigi$ in $\inve_i$ are the same as on $i,j,\sigi,\sigj$ (in that order) in both $\ie$ and $\ile$ in Definition~\ref{def-invstats}. (Recall that $\ile$ counts inversions among excedances where the rightmost of the two excedances is linked, so that ($\ie-\ile$) counts such inversions where the rightmost excedance is not linked.) Analogous considerations apply to $\inva_i$ and $\ninva_i$, whereas $\iefp_i$ will be applied only to fixed points $i$.

The structure of the Motzkin path $\eta(\sigma)$ closely depends on the type (linked vs.\ 
non-linked) of the excedances and anti-excedances in $\sigma$, as well as on the number of \emph{chains} formed by these. Namely:

\begin{definition}\label{def-chains-span}
A sequence $i_1,i_2,\ldots,i_p$ of excedances forms a \emph{chain of excedances} if
\begin{equation}
i_1<\sig(i_1)=i_2<\sig(i_2)=i_3<\cdots<\sig(i_{p-1})=i_p<\sig(i_p).
\end{equation}
If $i_1,i_2,\ldots,i_p$ is a maximal chain of excedances (w.r.t. containment), then its \emph{starter} is the excedance $i_1\mapsto \sig(i_1)$.  A chain of excedances  $i_1,i_2,\ldots,i_p$ is said to \emph{span} $z\in[n]$ if $i_1<z<\sig(i_p)$.

A sequence $i_1,i_2,\ldots,i_p$ of anti-excedances forms a \emph{chain of anti-excedances} if
\begin{equation}
i_1>\sig(i_1)=i_2>\sig(i_2)=i_3>\cdots>\sig(i_{p-1})=i_p>\sig(i_p).
\end{equation}
The \emph{starter} of a maximal such chain is the anti-excedance $i_1\mapsto \sig(i_1)$.  
A chain of anti-excedances $i_1,i_2,\ldots,i_p$ is said to \emph{span} $z\in[n]$ if $i_1>z>\sig(i_p)$.
\end{definition}

\begin{remark}\label{rem-start-span}
  Note that in a maximal chain of excedances all excedances are linked except for the starter, and likewise for a maximal chain of anti-excedances.  Also, $x\in[n]$ is spanned by a chain of excedances or anti-excedances if the chain `begins and `ends' on different sides of $x$.
\end{remark}


\begin{lemma}\label{lemma-Motzkin-steps}
For every $\sig\in \clsn$ and $m\in[n]$ the number of non-linked excedances in $[m]$ is at least as large as the number of non-linked anti-excedances in $[m]$. Moreover, the total number of non-linked excedances in $\sig$ is equal to the total number of non-linked anti-excedances.
\end{lemma}

\begin{proof}
Let $i$ be any non-linked anti-excedance in $\sig$, so $\sig(i)<i$. This implies that 
$j=\sig^{-1}(i)$ is an excedance, so $j<i$, since otherwise $i$ would be a linked anti-excedance.  Now, $j=\sig^{-1}(i)$ is an excedance whose maximal chain of excedances starts with a non-linked excedance $k\le j$.  We associate the non-linked anti-excedance $i$ to this  non-linked excedance $k$.

Two different maximal chains of excedances cannot intersect in any $i\in[n]$, since that would imply that $\sig(i)$ had two different values for some $i$.
Thus, each non-linked anti-excedance is associated to a unique non-linked excedance preceding it, which shows that reading from left to right we can never have more non-linked anti-excedances than non-linked excedances.

To show that a permutation $\sig$ has equally many non-linked excedances and non-linked anti-excedances in total, apply the above argument to the \emph{reverse complement} of $\sig$, whose two-line diagram is obtained by reflecting the diagram for $\sig$ in its vertical bisector (and reversing the numbers to have them increase from left to right).  That transformation clearly interchanges non-linked excedances and non-linked anti-excedances.
\end{proof}

\subsection{Constructing $\eta$}

To define the map $\eta:\clsn\to\motzn$ of Theorem~\ref{thm-bij}, we begin by associating to each permutation $\sig\in \clsn$ a path in $\R^2$ starting at $(0,0)$ and composed of level steps, NE 
upsteps, and SE downsteps as follows. For $i=1,\ldots,n$, the $i$th step in the path is an upstep if $i\mapsto \sig(i)$ is a non-linked excedance, a downstep if $i\mapsto \sig(i)$ is a non-linked anti-excedance, and a level step otherwise. By Lemma~\ref{lemma-Motzkin-steps}, the path will terminate on the abscissa without ever falling below it. The paths obtained are therefore Motzkin paths.

Fix $\sig\in \clsn$.  If an excedance $i\mapsto \sig(i)$ is linked, the $i$-th step is a level step labeled 
\begin{equation}\label{labels1}
a^{\inve_i(\sig)}b^{\ninve_i(\sig)}s
\end{equation}
and, otherwise, if the excedance is not linked, the $i$-th step is an upstep labeled
\begin{equation}\label{labels2}
c^{\inve_i(\sig)}d^{\ninve_i(\sig)}p.
\end{equation}

Similarly, if an anti-excedance $i\mapsto \sig(i)$ is linked, the $i$-th step is a level step labeled 
\begin{equation}\label{labels3}
f^{\inva_i(\sig)}g^{\ninva_i(\sig)}t
\end{equation}
and, otherwise, if the anti-excedance is not linked, the $i$-th step is a downstep labeled
\begin{equation}\label{labels4}
h^{\inva_i(\sig)}\ell^{\ninva_i(\sig)}r.
\end{equation}

Finally, if $i$ is a fixed point in $\sig$ then the $i$-th step in the path is a level step labeled 
\begin{equation}\label{labels5}
u\cd w^{\iefp_i(\sig)}.  
\end{equation}

The following identities are elementary:
\begin{equation}\label{eq-ile-nile}
\prod_{\substack{i\in [n]\\ \sigma^{-1}(i)<i<\sigma(i)}}a^{\inve_i(\sig)}b^{\ninve_i(\sig)}=a^{\ile(\sig)} b^{\nile(\sig)}
\end{equation}
and
\begin{equation}\label{eq-ie-nie}
\prod_{\substack{i\in [n]\\ i<\min(\sigma(i),\sigma^{-1}(i))}}c^{\inve_i(\sig)}d^{\ninve_i(\sig)}=c^{\ie(\sig)-\ile(\sig)} d^{\nie(\sig)-\nile(\sig)}.
\end{equation}

Analogous statements hold for the anti-excedances, and for the fixed
points (with $\iefp$), showing that the weight of a labeled path~$\eta(\sig)$
is as stated in the following theorem.

\begin{theorem}\label{main-thm-2}
  The map $\eta: \clsn \to \motzn$ is a bijection with the property that, for $\mp=\eta(\sig)$,
\begin{eqnarray*}
\wt(\mp)&=&\wtprod
\end{eqnarray*}
\end{theorem}

The proof of Theorem~\ref{main-thm-2} hinges on the following lemma. 
\begin{lemma}\label{lemma-maps-sets}
  Suppose $A$ and $B$ are finite subsets of $\NN$ with $\#A=\#B$ and let
  $\phi,\psi:A\rightarrow B$ be bijections such that, for every $i\in A$, we have
  $i<\phi(i)$, $i<\psi(i)$, and
\begin{equation}\label{eq-AB}
  \#\{x\in A\mid x<i<\phi(i)<\phi(x)\}=
  \#\{x\in A\mid x<i<\psi(i)<\psi(x)\}.
\end{equation}
Then $\phi=\psi$.
\end{lemma}
\begin{proof}
  We proceed by induction on $\kappa:=\#A=\#B$.  Let $A=\{i_1,\ldots,i_\kappa\}$ and
  $B=\{j_1,\ldots,j_\kappa\}$, with $i_1<\cdots<i_\kappa$ and
  $j_1<\cdots<j_\kappa$.  Then it is easy to verify that \eqref{eq-AB}
  implies $\phi(i_\kappa)=\psi(i_\kappa)=j_\lambda$ where
\begin{eqnarray*}
\lambda&=&\kappa-\#\{x\in A\mid x<i_\kappa<\phi(i_\kappa)<\phi(x)\}\\
&=&\kappa-\#\{x\in A\mid x<i_\kappa<\psi(i_\kappa)<\psi(x)\}.
 \end{eqnarray*}
Now removing $i_\kappa$ from $A$ and $j_\lambda$ from $B$, and restricting $\phi$ and $\psi$ accordingly, the equality in \eqref{eq-AB} still holds for the modified sets and bijections. Thus, by induction, $\phi=\psi$.
\end{proof}

\begin{proof}[Proof of Theorem~\ref{main-thm-2}]
  Since $|\motzn|=n!$ by Lemma~\ref{lemma-num-motz}, it suffices to show that $\eta$ is injective. To this end, let $\sig,\sig'\in \clsn$ be permutations mapped to the same labeled Motzkin path $\mp=\eta(\sig)=\eta(\sig')$, and we will show that this forces $\sig=\sig'$.

  Let $\U$ be the set of all positions of the level steps in $\mp$ whose label is of the form $uw^i$.  Let $\E_0$ be the set of all positions of the upsteps in $\mp$, $\E_1$ the set of all positions of the level steps in $\mp$ whose label is of the form $a^i b^j s$ and $\E=\E_0\cup\E_1$. Similarly, let $\A_0$ be the set of all positions of the downsteps in $\mp$, $\A_1$ the set of all positions of the level steps in $\mp$ whose label is of the form $f^i g^j t$ and $\A=\A_0\cup\A_1$. Notice that the sets $\U, \E_0, \E_1, \A_0$, $\A_1$ are disjoint and their union is $[n]$. Also, $\E_0$ equals the set of non-linked excedances in $\sig$ and $\sig'$, $\E_1$ equals the set of linked excedances in~$\sig$ and~$\sig'$, and similarly for $\A_0$, $A_1$ and anti-excedances, whereas $\U$ is the set of fixed points in~$\sig$ and~$\sig'$.

Recalling that under the map $\eta$ fixed points in the permutation map to $uw^i$-labeled level steps, we have that $\sig(i)=\sig'(i)=i$ for all $i\in \U$. Next we show that $\sig(\E)=\sig'(\E)$.  Note first that every linked excedance is the image, under $\sig$, of an excedance, so $\E_1$ is contained in $\sig(\E)$.  Also, a non-linked excedance cannot belong to $\sig(\E)$, so the remainder of $\sig(\E)$ consists of anti-excedances.  These are precisely all the non-linked anti-excedances, because they are all images of excedances, by definition, and the linked anti-excedances are not. 
The same holds for $\sig'$ and so we have that
$$
\sig(\E)=\sig'(\E)=\E_1\cup \A_0.
$$
Let $\E=\{i_1,\ldots,i_\kappa\}$ with $i_1<\cdots<i_\kappa$ and 
$\sig(\E)=\{j_1,\ldots,j_\kappa\}$ with $j_1<\cdots<j_\kappa$.  Now, since $\sig$ and $\sig'$ map to the same labeled Motzkin path we must have $\inve_i(\sig)=\inve_i(\sig')$ for all $i$. This implies that $\sig$ and $\sig'$ satisfy the conditions on $\phi$ and $\psi$ in Equation~\eqref{eq-AB}, with $A$ in Lemma~\ref{lemma-maps-sets} being~$\E$.
Thus, Lemma~\ref{lemma-maps-sets} implies that $\sig(\E)=\sig'(\E)$. (Note that $\ninve$ is the number of $x\in [n]$ such that $x<i_\kappa<\sig(x)<\sig(i_\kappa)$ and that the assignment $\sig(i_\kappa)=\sig'(i_\kappa)=j_{\kappa-\inve}$ implied by Lemma~\ref{lemma-maps-sets} is necessarily consistent with this.)

Finally, taking the reverse complements of $\sig$ and $\sig'$ or, equivalently, using the analogue of Lemma~\ref{lemma-maps-sets} in which all the inequalities are reversed, similarly shows that $\sig$ and $\sig'$ also agree on the set $\A$. 
Hence, $\sig=\sig'$, so $\eta$ is injective and thus bijective.
\end{proof}

\begin{remark}
  Note that in the proof of Theorem~\ref{main-thm-2}, no mention is made of the values of $\iefp_i(\sig)$.  The reason is that a permutation is determined uniquely by a much smaller set of statistics than those present in the labels of the corresponding Motzkin path.  It is easy to see that knowing the excedance bottoms and tops ($i$ and $\sig(i)$ for all excedances $i$), together with the set of fixed points and the vector statistics $\inve_i$ and $\inva_i$ for all $i$, is enough to determine a permutation.\end{remark}

\subsection{Some properties of $\eta$}

It follows from the definition of the map $\eta$ that the height of the left end of the $i$th step in the path $\eta(\sig)$ equals the number of non-linked excedances preceding the $i$th place in $\sig$ minus the number of preceding non-linked anti-excedances, since only these incur a height difference.  Below we give a different description of the height of the $i$th step, in terms of statistics derived from those in Definition~\ref{vector-stats}.

\begin{definition}\label{def-prex-fola}
  Given $\sig\in\clsn$ let
  \begin{itemize}
  \item $\prex_i(\sig)=\#\{x\in[n] \mid x<i<\sig(x)\}$,
  \item $\fola_i(\sig)=\#\{x\in[n] \mid \sig(x)<i<x\}$.
  \end{itemize}
\end{definition}
Note that if $i$ is an excedance then $\prex_i(\sig)=\inve_i(\sig)+\ninve_i(\sig)$ and if $i$ is an anti-excedance then $\fola_i(\sig)=\inva_i(\sig)+\ninva_i(\sig)$. If $i$ is a fixed point then $\prex_i(\sig)$ counts the same indices~$x$ as it does for an excedance.  
Note also that $\prex_i(\sig)$ is the number of excedances $x$ preceding $i$ such that $\sig(x)>i$ and that $\fola_i(\sig)$ is the number of anti-excedances $x$ following $i$ such that $\sig(x)<i$ .

\begin{proposition}
  Given $\sig\in\clsn$ let $\mp$ be the labeled Motzkin path $\mp=\eta(\sig)$.  Then, for $i\in[n]$,
\begin{enumerate}[label=(\roman*)]
\item if $i$ is a linked excedance then the $i$th step in $\mp$ is a level step at height
$\prex_i(\sig)+1$.

\item if $i$ is a non-linked excedance then the $i$th step in $\mp$ is an upstep starting at height $\prex_i(\sig)$.

\item if $i$ is a linked anti-excedance then the $i$th step in $\mp$ is a level step at height $\fola_i(\sig)+1$.

\item if $i$ is a non-linked anti-excedance then the $i$th step in $\mp$ is a downstep starting at height $\fola_i(\sig)$.

\item If $i$ is a fixed point then the $i$th step in $\mp$ is a level step whose height is $\prex_i(\sig)$.
\end{enumerate}
\label{lemma-labeling}
\end{proposition}

\begin{proof}
Recall that the starter of each maximal chain of excedances is a non-linked excedance, 
corresponding to an upstep, while any remaining excedances in the chain are necessarily 
linked, corresponding to level steps. Furthermore, each maximal chain of excedances terminates in a non-linked anti-excedance~$j$, corresponding to a downstep, namely, where $\sig^{-1}(j)$ is the last excedance in the chain. Conversely a non-linked anti-excedance always terminates some chain of excedances. 

Thus, the number of non-linked excedances (strictly) preceding an excedance $i$ equals the 
number of maximal chains of excedances starting before $i$, and the number of non-linked anti-excedances preceding $i$ equals the number of maximal chains of excedances that end 
before $i$.  The difference between these two numbers is the number of maximal chains of excedances that span $i$, but this difference is also the height of the $i$th step of $\mp$, as pointed out just before Definition~\ref{def-prex-fola}.  Each such chain spanning $i$ has a rightmost excedance $x$ strictly preceding $i$ if $i$ is a non-linked excedance, which thus satisfies $x<i<\sig(x)$, and those inequalities are satisfied by precisely one excedance $x$ in each maximal chain 
spanning $i$. The number of such $x$ is precisely $\prex_i(\sig)$, so if $i$ is a non-linked excedance the $i$th step starts at height $\prex_i(\sig)$.  If $i$ is a linked excedance then the maximal chain that $i$ belongs to also spans $i$, and so the $i$th step (which is level) is at height $\prex_i(\sig)+1$.

The argument in the case of anti-excedances is analogous, since reversing a path $\eta(\sig)$, which corresponds to taking the reverse complement of $\sig$ (see proof of  Lemma~\ref{lemma-Motzkin-steps}), turns excedances into anti-excedances and vice versa, and preserves linking.

Finally, if $i$ is a  fixed point then the number of chains spanning $i$ equals the 
number of elements  $x\in[n]$ satisfying $x<i<\sig(x)$, 
so the height of the corresponding level step is $\prex_i(\sig)$, the argument being identical to that for non-linked excedances. 
\end{proof}

\subsection{Ties to previous bijections}\label{sec-comparisons}

As mentioned before, several authors have defined statistics on permutations that have 
then been used to construct bijections to Motzkin paths, labeled to reflect those statistics~\cite{fravie,foata-zeilberger-denerts,biane-exc-inv-heine,simion-stanton-octabasic,ClarkeSteingrimssonZeng,randrianarivony-poly-ortho-sheffer}.  These statistics essentially come in two flavors, being based either on excedances, as in the present paper 
and~\cite{foata-zeilberger-denerts,biane-exc-inv-heine,randrianarivony-poly-ortho-sheffer} or descents as in~\cite{fravie,simion-stanton-octabasic}, a bijection in~\cite{ClarkeSteingrimssonZeng} translating between these.  In each of these two groups, each permutation is mapped to the same Motzkin path, but labels differ.  We give here, without proofs, some indication of the similarities and differences between these labelings, showing in the process that the bijection introduced in this paper cannot be recovered from previous published work.

First, it is helpful to consider a minor modification of the setup in~\cite{simion-stanton-octabasic}, where Simion and Stanton partition a permutation into \emph{ascent blocks} (which they call \emph{runs}), maximal contiguous increasing segments in a a permutation.  To facilitate the discussion here, we will follow~\cite{ClarkeSteingrimssonZeng} and partition into \emph{descent blocks}, which are decreasing rather than increasing.  A bijection translating the ascent-based framework in~\cite{simion-stanton-octabasic} to the descent-based one in~\cite{ClarkeSteingrimssonZeng} is taking \emph{reverse} of each permutation, that is, sending a permutation $a_1a_2\ldots a_n$ to $a_n\ldots a_2a_1$. For example, the reverse of the permutation $36\dd4\dd125$ is $521\dd4\dd63$, where we separate the ascent/descent blocks by dashes.  We then say that the first letter in a non-singleton block is an \emph{opener} and its last letter a \emph{closer}. Other letters in non-singleton blocks are \emph{insiders} (\emph{continuators} in~\cite{simion-stanton-octabasic}). In $521\dd4\dd63$ the openers are $5,6$, the closers are $1,3$, whereas~$2$ is an insider and $4$ a singleton.  Clearly, reversing a permutation interchanges openers with closers and preserves insiders and singletons, and the vector statistics $\lsg_i$ and $\rsg_i$~in\cite[Def.~2.1]{simion-stanton-octabasic} are also interchanged.  What we say about the statistics of Simion and Stanton below thus applies to the reversed version, with descent blocks instead of ascent blocks.

In~\cite[Section~3]{ClarkeSteingrimssonZeng} a bijection $\Phi$ on permutations is presented that translates the above `descent based' classes of letters into `excedance based' classes. Namely, in a permutation $\pi=a_1a_2\ldots a_n$, written in one line notation, if $a_i>i$ then $a_i$ is an `excedance top' and $i$ and `excedance bottom', whereas if $a_i\le i$ then $a_i$ is a `non-excedance top' and $i$ a `non-excedance bottom'.  When $\Phi$ is applied to a permutation, such as $521\dd3\dd6\dd74$, which $\Phi$ maps to 2715364, we get the following correspondences between these four classes, where the last column gives the classes used in the present paper (the singletons 3 and 6 mapping to 
a linked anti-excedance and a fixed point, respectively):

\setlength{\topsep}{-1em}
\begin{center}
\renewcommand{\arraystretch}{1.5}
\setlength{\tabcolsep}{1.2em} 
\begin{tabular}{ccc}
\mbox{in $\sig$ \cite{simion-stanton-octabasic}} & \mbox{in $\Phi(\sig)$ \cite{ClarkeSteingrimssonZeng}} & \mbox{in $\Phi(\sig)$ (present)}\\\hline
\mbox{opener} & \mbox{excedance top \& non-excedance bottom} & \mbox{non-linked anti-excedance}\\\hline
\mbox{closer} & \mbox{non-excedance top \& excedance bottom} & \mbox{non-linked excedance}\\\hline
\mbox{insider} & \mbox{excedance top \& excedance bottom} & \mbox{linked excedance}\\\hline
\mbox{singleton} & \mbox{non-excedance top \& non-excedance bottom} & \mbox{linked anti-excedance/fixpt}\\\hline
\end{tabular}\\[4mm]
\end{center}
\vskip1em

It turns out that all the bijections in the papers mentioned above translate a permutation to the same \emph{unlabeled} Motzkin path, modulo the bijection $\Phi$ and/or some trivial bijections such as reversing a permutation.  The bijection $\Phi$ also translates between the various statistics studied in~\cite{simion-stanton-octabasic} and~\cite{ClarkeSteingrimssonZeng}, respectively, where the inversion statistics in the latter correspond to the statistic $\rsg$ in the former, which counts occurrences of the 
vincular pattern $2\dd31$.  These statistics determine the labeling of the steps in the Motzkin paths that respectively correspond to permutations in each case. 
In~\cite{simion-stanton-octabasic} Simion and Stanton also incorporate the `mirror' statistic $\lsg$, which counts occurrences of $31\dd2$.  
This latter statistic can be incorporated in Theorem~10 in~\cite{ClarkeSteingrimssonZeng} 
(as was done in~\cite[Theorem~22]{cla-mans-count}), by replacing the brackets $[\cdot]_p$ with $[\cdot]_{p,p'}$, and including the corresponding mirror images of the inversion statistics in $f(x,p,q)$ in~\cite{ClarkeSteingrimssonZeng}.

Since fixed points are treated as non-excedance tops and non-excedance bottoms in~\cite{ClarkeSteingrimssonZeng}, and thus not distinguished from linked anti-excedances, both correspond to singletons in $\Phi^{-1}(\sig)$.  
 This setup thus differs from that of the present paper and in particular leads to the asymmetry in the continued fractions   that `mix' the brackets $[n]$ and $[n+1]$ in both~\cite{simion-stanton-laguerre} and~\cite{ClarkeSteingrimssonZeng}.  As Simion and Stanton~\cite{simion-stanton-laguerre} 
base their statistics on the one-line notation of ascent blocks, and classify a letter of a permutation according to the relative sizes of its predecessor and follower, it is impossible to accommodate more than four distinct classes of letters.

 There is nevertheless a way to characterize what happens to fixed points in a permutation 
$\sig$ under $\Phi^{-1}$.  Namely, a fixed point $i$ in $\sig$ corresponds in 
$\Phi^{-1}(\sig)$ to a singleton $i$ with $\lsg(i)=0$, that is, a singleton that is not the $2$ in any occurrence of $31\dd2$.  This corresponds to the fact that a fixed point in $\sig$ can not form an inversion with a non-excedance preceding it.

In~\cite{randrianarivony-poly-ortho-sheffer}, Randrianarivony presents a generating function that expands the one in~\cite{simion-stanton-octabasic} by adding five statistics: $\orsg$, which is evaluated on each of the four types of letters (openers, closers, etc.), and $D$. The statistic $D$ is a linear combination of other statistics in the generating function, and so does not add information.
The $\orsg$ statistics, on the other hand, are not vector statistics, since they count, for each type of letter, the number of such letters $i$ in $\sig$ with $\rsg_i(\sig)=0$. Because of our reversal here of the Simion-Stanton setup, our statistic $\lsg_i$ corresponds to $\rsg_i$ in the paper by Randrianarivony.  Thus, his pointing out that 
$\orsg$ applied to the singletons counts double descents~$i$ with $\rsg(i)=0$, is equivalent to our claim about fixed points in the preceding paragraph, because a double descent in the ascent-based setup is a singleton, between two ascent blocks.

As for the bijections of Fran\c{c}on and Viennot~\cite{fravie} and Foata and Zeilberger~\cite{foata-zeilberger-denerts}, it is explained in~\cite[Section~5]{ClarkeSteingrimssonZeng} how they are equivalent, via the bijection~$\Phi$.

In short, separating fixed points as a class of their own --- as we have done here --- in 
defining the vector statistics at the heart of these different schemes, is not possible within the schemes of Simion-Stanton~\cite{simion-stanton-octabasic} and Randrianarivony~\cite{randrianarivony-poly-ortho-sheffer}. The same is true of the other schemes mentioned above, which all are equivalent to that of Simion and Stanton, as previously explained.

In the recent preprint~\cite{sokal-zeng-masterpolys} Sokal and Zeng consider both `records' (left-to-right maxima) and  `antirecords' (right-to-left minima), and partition letters in a permutation into four classes depending on these attributes.  They also partition letters into `cycle peaks', `cycle valleys', `cycle double rises', `cycle double falls' and fixed points, the double rises and double falls corresponding to our linked excedances and linked anti-excedances, respectively, the other two `cycle' classes to our non-linked cases of these.  By considering intersections of these two kinds of classes, they come up with a partitioning into ten distinct categories. In addition to this they incorporate nine statistics that are refinements of the crossings and nestings defined by Corteel~\cite{Corteel2007}, combinations of which can be shown to equal our inversion statistics in items~\ref{firststat}-\ref{laststat}, Definition~\ref{def-invstats}.

  It is worth comparing our labels of steps in  Motzkin paths, in \eqref{labels1}--\eqref{labels5} above, with the corresponding labels of the bijection of Foata and Zeilberger~\cite{foata-zeilberger-denerts}, as described in~\cite[Section~5]{ClarkeSteingrimssonZeng}.  In~\cite{ClarkeSteingrimssonZeng} the label assigned to each excedance $i$ is the number of excedances $j$ preceding $i$ where $\sig(j)>\sig(i)$, that is,
  the label of an excedance is the `inversion bottom number' of that excedance among all excedances.  This agrees exactly with our label $\inve_i(\sig)$, which is independent of whether the excedance $i$ is linked.
In \cite{ClarkeSteingrimssonZeng}, the label assigned to each non-excedance (which includes fixed points), is 
the number of non-excedances following it and smaller than it, which agrees with our $\inva$ only in the case of anti-excedances.

While a fixed point is classified as a non-excedance in~\cite{ClarkeSteingrimssonZeng}, and 
the label associated to it therefore depends on smaller non-excedances following it, here we label a fixed point $i$ by $uw^{\iefp_i(\sig)}$, which thus depends on excedances preceding it.  The labeling in~\cite{ClarkeSteingrimssonZeng} consequently differs from ours.    Our labeling is also different from Biane's labeling in~\cite{biane-exc-inv-heine}, which in~\cite{ClarkeSteingrimssonZeng} is shown to be closely related to the labeling described in~\cite{ClarkeSteingrimssonZeng}.  Finally, the  labeling of Simion-Stanton in~\cite{simion-stanton-laguerre}, after the translation applying $\Phi$  to the reverse of each permutation in their setting, is also different from ours, as is inevitable since they do not distinguish between what after that translation become linked anti-excedances and fixed points, respectively.

Crucially, our labeling of excedances and anti-excedances is independent of fixed points in a permutation, which is not the case for the labeling in~\cite{ClarkeSteingrimssonZeng} nor, by the equivalences we have described above, for the other bijections 
mentioned above (except for the recent preprint~\cite{sokal-zeng-masterpolys}).  
Specifically, underlying all these results is some scheme for partitioning letters that form a permutation, which is at the core of each construction. While the aforementioned papers (again excluding~\cite{sokal-zeng-masterpolys}) can only accommodate a partition into four classes, we have a separate, fifth class, containing the fixed points. This distinction allows for a greater variety of special cases, extensions to signed and colored permutations, as well as to a natural notion of $k$-arrangements.

\section{Acknowledgements}
We would like to thank Roman Stocker and Martin Ackermann for their hospitality at ETH Z\"urich where some of this work was done. This collaboration started through a LMS Celebrating New Appointments Grant and was further enabled by the Lancaster Mathematics \& Statistics Visitor Fund. Thanks also to Alan Sokal and Jiang Zeng, for pointing out the paper by Randrianarivony~\cite{randrianarivony-poly-ortho-sheffer}, and for kindly sharing a draft of their manuscript~\cite{sokal-zeng-masterpolys} with us as ours was being finalized. Lastly, we are indebted to an anonymous referee whose careful reading and valuable suggestions improved this paper, especially with regard to the streamlining of the proof of Theorem~\ref{main-thm-2} via Lemma~\ref{lemma-maps-sets}.


\def\aboveex{1.8ex}
\def\belowex{-2ex}

\begin{landscape}
\oddsidemargin1ex
\begin{table}
\setlength{\tabcolsep}{3ex}
\small
\begin{tabular}{llll}
Parameter settings&Combinatorial objects&\parbox{12em}{
Moment seq.
(\sc oeis \cite{OEIS})
}&Measure \\[\aboveex]\hline\\[\belowex] 

&Permutations & $n!$\,\, (A000142)& Exponential: $\displaystyle e^{-x} \mathbbm 1_{[0,\infty)}\,dx$ \\[\aboveex]\hline\\[\belowex] 

\parbox{8em}{$h,s,t,u=0$}&
\parbox{8em}{Perfect matchings}&
\parbox{12em}{
$(2n-1)!!$\,\,\,\, (A001147)
}&
\parbox{8em}{Gaussian$^\ast$:$\frac{1}{\sqrt{2\pi}}e^{-x^2/2}\,dx$} \\[\aboveex]\hline\\[\belowex] 

\parbox{12em}{$h,s,t,u=0;\;\; c=q$}&
\parbox{14em}{Perfect matchings by \#crossings }&
\parbox{12em}{
$\disp\sum_{\pi\in \mathcal P_2(2n)}q^{\text{cr}(\pi)}$ (A067311)
}&
\parbox{12em}{$q$-Gaussian$^\ast$ \cite{Bozejko1991, Speicher1992}} \\[\aboveex]\hline\\[\belowex] 

\parbox{12em}{$h,s,t,u=0;\;\; c=q;\;\,\, d=t$}&
\parbox{12em}{Perfect  matchings by\\ \#crossings \& nestings}&
\parbox{8em}{$\disp\sum_{\pi\in \mathcal P_2(2n)}\!\!\!\!q^{\text{cr}(\pi)}t^{\text{nest}(\pi)}$}&
\parbox{12em}{$(q,t)$-Gaussian$^\ast$ \\\cite{Blitvic2012, Blitvic2014}} \\[\aboveex]\hline\\[\belowex] 

\parbox{8em}{$c,h,s,t,u=0$}&
\parbox{12em}{ Non-crossing set partitions}&
\parbox{8em}{
$\frac{1}{n+1}{2n\choose n}$\,\,\,\,(A000108) Catalan numbers 
}&
\parbox{9em}{Wigner semicircle* \\$\frac{1}{2\pi}\sqrt{4-x^2}\mathbbm 1_{[-2,2]}\,dx$} \\[\aboveex]\hline\\[\belowex] 

\parbox{8em}{$h,t=0;\;\;\; p,u=\lambda$} &
\parbox{12em}{Set partitions by \#blocks} &
\parbox{12em}{
Stirling $2^{nd}$: $\sum_{\pi \in \mathcal P(n)}\lambda^{|\pi|}$ (A008277)
} &
\parbox{12em}{Poisson, rate $\lambda$: \,\,\,\,$\displaystyle e^{-\lambda}\lambda^k/k!$} \\[\aboveex]\hline\\[\belowex] 

\parbox{10em}{$a,c,h,t=0,\;\; p,u=\lambda$}&
\parbox{12em}{Non-crossing set partitions of $[n]$ into $k$ blocks}&
\parbox{8.5em}{
$\sum_{k} \frac{1}{n}{n\choose k}{n\choose k-1}\lambda^k$ Narayana numbers (A001263)
}&
\parbox{12em}{Marchenko-Pastur:\\ $\lambda_{\pm}=(1\pm\sqrt\lambda)^2,\lambda\geq1$, \\
$\frac{\sqrt{(\lambda_+-x)(x-\lambda_-)}}{2\pi x}{\mathbbm1}_{[\lambda_-,\lambda_+]}\,dx$} \\[4.5ex]\hline\\[\belowex] 

\parbox{12em}{$h,t=0;\;\; a,c=q;\;\;p,u=\lambda$ }&
\parbox{9em}{Restricted crossings\\in partitions \cite{Biane1997}}&
\parbox{9em}{
$\disp\sum_{\pi \in \mathcal P(n)}q^{\text{cr}(\pi)}\lambda^{|\pi|}$
}&
\parbox{12em}{$q$-Poisson, rate $\lambda$\,\cite{Anshelevich2001}} \\[\aboveex]\hline\\[\belowex] 

\parbox{13em}{$h,t,u=0;\;\; b,d=x;\;\;a,c=q$ }&
\parbox{12em}{Restricted cross/nest\\ in partitions \cite{Kasraoui2006}}&
\parbox{6em}{
$\displaystyle\sum_{\pi \in \mathcal P(n)} q^{\text{cr}(\pi)}x^{\text{nest}(\pi)}$
}&
\parbox{9em}{$(q,t)$-Poisson \cite{Ejsmont}} \\[\aboveex]\hline\\[\belowex] 

$u=0$&Derangements&A000166&e.g. \cite{Martin2015} \\[.7ex]\hline\\[\belowex] 

$s,t,u=0$ & \parbox{14em}{Alternating permutations of $[2n]$}&A000364&e.g. \cite{Sokal2018}*
\\[1ex]\hline\\[\belowex] 

$a,c,f,h=0;\;\;p\!=\!2$ & Little Schr\"oder numbers& A001003 & \cite{Mlotowski2013} \\[1ex]\hline\\[\belowex] 

$a,u=0;\;\;t=2$ & 
\parbox{16em}{Permutations, no strong fixed points} & A052186 & \cite{Martin2015} \\[1ex]\hline\\[\belowex] 

$p,s=x$ & Eulerian polynomials &$\sum_{\sigma\in\clsn} x^{\des(\sigma)}$ (A008292)& \cite{Barry2018,Borowiec2016} \\[\aboveex]\hline\\[\belowex] 

\parbox{12em}{$p,s=2x;\;\; r,t=2;\;\;\\ u=x+1$}&
\parbox{12em}{Eulerian polynomials for\\ hyperoctahedral groups} &$\sum_{\sigma\in B_n} x^{\des(\sigma)}$ (A060187)& \cite{Barry2018,Borowiec2016} \\[1.5ex]\hline\\[\belowex]
\end{tabular}
  \caption{\label{table-measures}Examples of moment sequences encoded in generating function \cf\ in Definition~\ref{def-gz} with their corresponding probability laws. 
All other parameters, unless otherwise indicated, are set to $1$.  $\mathcal P(n)$ is set partitions of $[n]$, $\mathcal P_2(2n)$ perfect matchings. Measures marked by an asterisk ($\,^\ast$) have vanishing odd moments; in this case the sequences shown are the even moments. Note that `suppressing' zeros from the odd locations of a moment sequence amounts to squaring the underlying random variable. Conversely, for measures supported on $[0,\infty)$, inserting zeros in the odd locations amounts to taking symmetric (randomized) square roots.}
\end{table}
\end{landscape}


\begin{table}[h]
\small
\setlength{\tabcolsep}{0pt}
\begin{tabular}{lll}
\parbox{12em}{Orthogonal\\ polynomial sequence}&\parbox{5em}{Normaliz.\\recurr.} &Parameters $(a,b,c,d,f,g,h,\ell,p,r,s,t,u)$ in \cf\ with $w=0$ in all cases\\[2ex]
\hline\\[-2ex]
Discrete $q$-Hermite II&(14.29.5)&$(0,0,0,q^{-2},0,0,1,q,q^{-1},1-q,0,0,0)$\\

Discrete $q$-Hermite I & (14.28.4) & $(0,0,0,q,0,0,1,q,1,1-q,0,0,0)$\\

Stieltjes-Wigert&(14.27.4)&$(0,q^{-2},0,q^{-4},0,q^{-1},1,q,q^{-3},1-q,(1+q)q^{-3},-q^{-1},q^{-1})$\\

Continuous $q$-Hermite&(14.26.4)&$(0,0,0,1,0,0,1,q,1/4,1-q,0,0,0)$\\

Al-Salam-Carlitz II& (14.25.4)&$(0,q^{-1},0,q^{-2},0,0,1,q,a q^{-1},1-q,(a+1)q^{-1},0,a+1)$\\

Al-Salam-Carlitz I& (14.24.4)&$(0,q,0,q,0,0,0,q,a,1-q,(a+1)q,0,a+1)$\\

$q$-Charlier ($a$=1)&(14.23.4)&$(0,q^{-2},0,q^{-4},0,q^{-2},1,q^{2},q^{-3},1-q^2,q^{-3},q^{-2},1+q^{-1})$\\

$q$-Laguerre ($\alpha$=0)&(14.21.6)&$(0,q^{-2},q^{-4},q^{-3},0,q^{-1},1,q,q^{-3}$--$q^{-2},1$--$q,(1$+$q)q^{-3},$-$2q^{-1},(1$-$q)q^{-1})$\\

Little $q$-Lag./Wall ($a$=1)&(14.20.4)&$(0,q,q,q^2,0,q^2,q,q^2,1-q,q(1-q),2q,-(1+q)q^2,1-q)$\\

Cont. $q$-Laguerre ($\alpha$=0)&(14.19.4)&$(0,q,1,q,0,0,1,q,(1$--$q)/2,(1$--$q)/2,(1$+$\sqrt q)q^{5/4}/2,0,q^{1/4}(1$+$q^{1/2})/2)$\\

Cont. big $q$-Hermite&(14.18.5)&$(0,q,1,q,0,0,1,0,(1-q)/4,1,aq/2,0,a/2)$\\
$q$-Meixner ($b$=$c$=1)&(14.13.4)&$(0,q^{-2},q^{-4},q^{-3},q^{-2},q^{-1},1,q^2,q^{-3}\!-q^{-2},1\!-q^2,q^{-3},q^{-2}\!-q^{-1},q^{-1})$\\

Big $q$-Laguerre ($a$=$-b$=1) &(14.11.4)&$(0,q^2,q,q^2,0,q,1,q^2,q^2(1-q),1-q^2,q^3(1+q),-q^2,q^2)$\\
$q$-Meixner-Pollaczek($a$=$\sqrt{q}$)&(14.9.4)&$(0,q,1,q,0,0,1,q,(1-q)/2,(1-q)/2,q^{3/2}\cos(\phi),0,\sqrt{q}\cos(\phi))$\\

Al-Salam-Chihara ($ab$=$q$)&(14.8.5)&$(0,q,1,q,0,0,1,q,(1-q)/2,(1-q)/2,(a+b)q/2,0,(a+b)/2)$\\

Continuous dual $q$-Hahn~~~~~~&(14.3.5)&$(0,q^2,1,q^2,0,q^2,1,q^2,(1-q^2)/2,(1-q^2)/2,a^{-1}q^4/2,aq^2/2,(a^2+q^2)/2a)$\\
 ($ab$=$-q$, $ac$=$q$, $bc$=$-q$)
\end{tabular}
\caption{\label{table-ops}
  Examples of orthogonal polynomial families in $q$-Askey scheme  encompassed by~$\cf$ in Definition~\ref{def-gz}. References for normalized recurrences are to equations in~\cite{koekoek-hyp-ortho-poly-q}. Choice of parameters in \cf\ is generally non-unique.}
\end{table}


\bibliographystyle{alpha}
\bibliography{BSrefs}

\begin{thebibliography}{dMSW95}

\bibitem[Ans01]{Anshelevich2001}
Michael Anshelevich.
\newblock Partition-dependent stochastic measures and {$q$}-deformed cumulants.
\newblock {\em Doc. Math.}, 6:343--384 (electronic), 2001.

\bibitem[Bar18]{Barry2018}
Paul Barry.
\newblock On a transformation of {R}iordan moment sequences.
\newblock {\em J. Integer Seq.}, 21(7):Art. 18.7.1, 19, 2018.

\bibitem[BBEP17]{bevan-al-staircase-1324}
David Bevan, Robert Brignall, Andrew {Elvey Price}, and Jay Pantone.
\newblock Staircases, dominoes, and the growth rate of 1324-avoiders.
\newblock {\em Electron. Notes Disc. Math.}, (61):123--129, 2017.

\bibitem[BEH15]{Bozejko2015}
Marek Bo\.{z}ejko, Wiktor Ejsmont, and Takahiro Hasebe.
\newblock Fock space associated to {C}oxeter groups of type {B}.
\newblock {\em J. Funct. Anal.}, 269(6):1769--1795, 2015.

\bibitem[BEH17]{Bozejko2017}
Marek Bo\.{z}ejko, Wiktor Ejsmont, and Takahiro Hasebe.
\newblock Noncommutative probability of type {$D$}.
\newblock {\em Internat. J. Math.}, 28(2):1750010, 30, 2017.

\bibitem[BGM14]{bagno-garber-mansour-color-perms}
Eli Bagno, David Garber, and Toufik Mansour.
\newblock On the group of alternating colored permutations.
\newblock {\em Electron. J. Combin.}, 21(2):Paper 2.29, 28, 2014.

\bibitem[Bia93]{biane-exc-inv-heine}
Philippe Biane.
\newblock Permutations suivant le type d'exc\'{e}dance et le nombre
  d'inversions et interpr\'{e}tation combinatoire d'une fraction continue de
  {H}eine.
\newblock {\em European J. Combin.}, 14(4):277--284, 1993.

\bibitem[Bia97]{Biane1997}
Philippe Biane.
\newblock Some properties of crossings and partitions.
\newblock {\em Discrete Math.}, 175(1-3):41--53, 1997.

\bibitem[Bli12]{Blitvic2012}
Natasha Blitvi\'{c}.
\newblock The {$(q,t)$}-{G}aussian process.
\newblock {\em J. Funct. Anal.}, 263(10):3270--3305, 2012.

\bibitem[Bli14]{Blitvic2014}
Natasha Blitvi\'{c}.
\newblock Two-parameter non-commutative central limit theorem.
\newblock {\em Ann. Inst. Henri Poincar\'{e} Probab. Stat.}, 50(4):1456--1473,
  2014.

\bibitem[BM16]{Borowiec2016}
Anna Borowiec and Wojciech M{\l}otkowski.
\newblock New {E}ulerian numbers of type {$D$}.
\newblock {\em Electron. J. Combin.}, 23(1):Paper 1.38, 13, 2016.

\bibitem[BS91]{Bozejko1991}
Marek Bo\.{z}ejko and Roland Speicher.
\newblock An example of a generalized {B}rownian motion.
\newblock {\em Comm. Math. Phys.}, 137(3):519--531, 1991.

\bibitem[BS94]{Bozejko1994}
Marek Bo{\.z}ejko and Roland Speicher.
\newblock Completely positive maps on {C}oxeter groups, deformed commutation
  relations, and operator spaces.
\newblock {\em Math. Ann.}, 300(1):97--120, 1994.

\bibitem[CGZJ18]{conway-guttmann-zinn-1324}
Andrew~R. Conway, Anthony~J. Guttmann, and Paul Zinn-Justin.
\newblock 1324-avoiding permutations revisited.
\newblock {\em Adv. in Appl. Math.}, 96:312--333, 2018.

\bibitem[Cla01]{cla-gpa}
Anders Claesson.
\newblock Generalized pattern avoidance.
\newblock {\em European J. Combin.}, 22(7):961--971, 2001.

\bibitem[CM02]{cla-mans-count}
Anders Claesson and Toufik Mansour.
\newblock Counting occurrences of a pattern of type {$(1,2)$} or {$(2,1)$} in
  permutations.
\newblock {\em Adv. in Appl. Math.}, 29(2):293--310, 2002.

\bibitem[Com74]{comtet-advanced}
Louis Comtet.
\newblock {\em Advanced combinatorics}.
\newblock D. Reidel Publishing Co., Dordrecht, enlarged edition, 1974.
\newblock The art of finite and infinite expansions.

\bibitem[Cor07]{Corteel2007}
Sylvie Corteel.
\newblock Crossings and alignments of permutations.
\newblock {\em Adv. in Appl. Math.}, 38(2):149--163, 2007.

\bibitem[CSSW12]{CorteelWilliams2}
S.~Corteel, R.~Stanley, D.~Stanton, and L.~Williams.
\newblock Formulae for {A}skey-{W}ilson moments and enumeration of staircase
  tableaux.
\newblock {\em Trans. Amer. Math. Soc.}, 364(11):6009--6037, 2012.

\bibitem[CSZ97]{ClarkeSteingrimssonZeng}
Robert~J. Clarke, Einar Steingr{\'{\i}}msson, and Jiang Zeng.
\newblock New {E}uler-{M}ahonian statistics on permutations and words.
\newblock {\em Adv. in Appl. Math.}, 18(3):237--270, 1997.

\bibitem[CW11]{CorteelWilliams1}
Sylvie Corteel and Lauren~K. Williams.
\newblock Tableaux combinatorics for the asymmetric exclusion process and
  {A}skey-{W}ilson polynomials.
\newblock {\em Duke Math. J.}, 159(3):385--415, 2011.

\bibitem[dMSW95]{medicis-stanton-white-charlier}
Anne de~M\'{e}dicis, Dennis Stanton, and Dennis White.
\newblock The combinatorics of {$q$}-{C}harlier polynomials.
\newblock {\em J. Combin. Theory Ser. A}, 69(1):87--114, 1995.

\bibitem[dMV94]{medicis-viennot-laguerre}
Anne de~M\'{e}dicis and Xavier~G. Viennot.
\newblock Moments des {$q$}-polyn\^{o}mes de {L}aguerre et la bijection de
  {F}oata-{Z}eilberger.
\newblock {\em Adv. in Appl. Math.}, 15(3):262--304, 1994.

\bibitem[Egg15]{DefyingGod}
Eric~S. Egge.
\newblock Defying {G}od: the {S}tanley-{W}ilf conjecture, {S}tanley-{W}ilf
  limits, and a two-generation explosion of combinatorics.
\newblock In {\em A century of advancing mathematics}, pages 65--82. Math.
  Assoc. America, Washington, DC, 2015.

\bibitem[Ehr00]{ehrenborg-hankel}
Richard Ehrenborg.
\newblock The {H}ankel determinant of exponential polynomials.
\newblock {\em Amer. Math. Monthly}, 107(6):557--560, 2000.

\bibitem[Ejs20]{Ejsmont}
Wiktor Ejsmont.
\newblock {Poisson type operators on the Fock space of type B and in the
  Blitvi\'{c} model}.
\newblock {\em Journal of Operator Theory}, 84(1):67--97, 2020.

\bibitem[Eli17]{elizalde-cf-fpsac}
Sergi Elizalde.
\newblock Continued fractions for permutation statistics.
\newblock {\em Discrete Math. Theor. Comput. Sci.}, 19(2):Paper No. 11, 24,
  2017.

\bibitem[EN03]{elizalde-noy}
Sergi Elizalde and Marc Noy.
\newblock Consecutive patterns in permutations.
\newblock {\em Adv. in Appl. Math.}, 30(1-2):110--125, 2003.
\newblock Formal power series and algebraic combinatorics (Scottsdale, AZ,
  2001).

\bibitem[EP18]{elvey-price-thesis}
Andrew Elvey-Price.
\newblock Selected problems in enumerative combinatorics: permutation classes,
  random walks and planar maps.
\newblock Thesis (Ph.D.)--University of Melbourne, 2018.

\bibitem[FHL20]{fu-han-lin-k-arrangements}
Shishuo Fu, Guo-Niu Han, and Zhicong Lin.
\newblock {$k$}-{A}rrangements, {S}tatistics, and {P}atterns.
\newblock {\em SIAM J. Discrete Math.}, 34(3):1830--1853, 2020.

\bibitem[Fla80]{Flajolet1980}
P.~Flajolet.
\newblock Combinatorial aspects of continued fractions.
\newblock {\em Discrete Math.}, 32(2):125--161, 1980.

\bibitem[FS70]{foata-schutz-eulerien}
Dominique Foata and Marcel-P. Sch\"{u}tzenberger.
\newblock {\em Th\'{e}orie g\'{e}om\'{e}trique des polyn\^{o}mes
  eul\'{e}riens}.
\newblock Lecture Notes in Mathematics, Vol. 138. Springer-Verlag, Berlin-New
  York, 1970.

\bibitem[FV79]{fravie}
Jean Fran\c{c}on and G\'erard Viennot.
\newblock Permutations selon leurs pics, creux, doubles mont\'ees et double
  descentes, nombres d'{E}uler et nombres de {G}enocchi.
\newblock {\em Discrete Math.}, 28(1):21--35, 1979.

\bibitem[FZ90]{foata-zeilberger-denerts}
Dominique Foata and Doron Zeilberger.
\newblock Denert's permutation statistic is indeed {E}uler-{M}ahonian.
\newblock {\em Stud. Appl. Math.}, 83(1):31--59, 1990.

\bibitem[Ham20]{Hamburger}
Hans Hamburger.
\newblock \"{U}ber eine {E}rweiterung des {S}tieltjesschen {M}omentenproblems.
\newblock {\em Math. Ann.}, (I) 81: 235-319, 1920; (II) 82: 120–164, 1920;
  (III) 82: 168-187, 1921, 1920.

\bibitem[IKZ13]{ismail-kasraoui-zeng-orthogonal}
Mourad E.~H. Ismail, Anisse Kasraoui, and Jiang Zeng.
\newblock Separation of variables and combinatorics of linearization
  coefficients of orthogonal polynomials.
\newblock {\em J. Combin. Theory Ser. A}, 120(3):561--599, 2013.

\bibitem[ISV87]{ismail-stanton-viennot-comb-hermite-askey}
Mourad E.~H. Ismail, Dennis Stanton, and G\'{e}rard Viennot.
\newblock The combinatorics of {$q$}-{H}ermite polynomials and the
  {A}skey-{W}ilson integral.
\newblock {\em European J. Combin.}, 8(4):379--392, 1987.

\bibitem[Jun03]{Junod2003}
Alexandre Junod.
\newblock Hankel determinants and orthogonal polynomials.
\newblock {\em Expo. Math.}, 21(1):63--74, 2003.

\bibitem[Kit07]{kitaev-intro-partord-patts}
Sergey Kitaev.
\newblock Introduction to partially ordered patterns.
\newblock {\em Discrete Appl. Math.}, 155(8):929--944, 2007.

\bibitem[Kit11]{kitaev-book}
Sergey Kitaev.
\newblock {\em Patterns in Permutations and Words}.
\newblock Monographs in Theoretical Computer Science. Springer-Verlag, 2011.

\bibitem[KLS10]{koekoek-hyp-ortho-poly-q}
Roelof Koekoek, Peter~A. Lesky, and Ren\'{e}~F. Swarttouw.
\newblock {\em Hypergeometric orthogonal polynomials and their
  {$q$}-analogues}.
\newblock Springer Monographs in Mathematics. Springer-Verlag, Berlin, 2010.
\newblock With a foreword by Tom H. Koornwinder.

\bibitem[Knu98]{taocp-3}
Donald~E. Knuth.
\newblock {\em The art of computer programming. {V}ol. 3}.
\newblock Addison-Wesley, Reading, MA, 1998.
\newblock Sorting and searching, Second edition [of MR0445948].

\bibitem[Kra99]{krattenthaler-adv-det-calc-slc}
C.~Krattenthaler.
\newblock Advanced determinant calculus.
\newblock {\em S\'{e}m. Lothar. Combin.}, 42:Art. B42q, 1999.
\newblock The Andrews Festschrift (Maratea, 1998).

\bibitem[Kra05]{krattenthaler-adv-det-calc-complement}
C.~Krattenthaler.
\newblock Advanced determinant calculus: a complement.
\newblock {\em Linear Algebra Appl.}, 411:68--166, 2005.

\bibitem[KSZ11]{kasraoui-stanton-zeng-salam-chihara-laguerre}
Anisse Kasraoui, Dennis Stanton, and Jiang Zeng.
\newblock The combinatorics of {A}l-{S}alam--{C}hihara {$q$}-{L}aguerre
  polynomials.
\newblock {\em Adv. in Appl. Math.}, 47(2):216--239, 2011.

\bibitem[KZ06]{Kasraoui2006}
Anisse Kasraoui and Jiang Zeng.
\newblock Distribution of crossings, nestings and alignments of two edges in
  matchings and partitions.
\newblock {\em Electron. J. Combin.}, 13(1):Research Paper 33, 12 pp.
  (electronic), 2006.

\bibitem[Mac04]{macmahon}
Percy~A. MacMahon.
\newblock {\em Combinatory analysis. {V}ol. {I}, {II} (bound in one volume)}.
\newblock Dover Phoenix Editions. Dover Publications Inc., Mineola, NY, 2004.
\newblock Reprint of {{\i}t An introduction to combinatory analysis} (1920) and
  {{\i}t Combinatory analysis. Vol. I, II} (1915, 1916).

\bibitem[MK15]{Martin2015}
Richard~J. Martin and Michael~J. Kearney.
\newblock Integral representation of certain combinatorial recurrences.
\newblock {\em Combinatorica}, 35(3):309--315, 2015.

\bibitem[M{\l}o10]{Mlotowski2010Documenta}
Wojciech M{\l}otkowski.
\newblock Fuss-{C}atalan numbers in noncommutative probability.
\newblock {\em Doc. Math.}, 15:939--955, 2010.

\bibitem[M{\l}o12]{Mlotowski2012Colloquium}
Wojciech M{\l}otkowski.
\newblock Probability measures corresponding to {A}val numbers.
\newblock {\em Colloq. Math.}, 129(2):189--202, 2012.

\bibitem[MP13]{Mlotowski2013}
Wojciech M{\l}otkowski and Karol~A. Penson.
\newblock The probability measure corresponding to 2-plane trees.
\newblock {\em Probab. Math. Statist.}, 33(2):255--264, 2013.

\bibitem[MP14]{Mlotowski2014}
Wojciech M{\l}otkowski and Karol~A. Penson.
\newblock Probability distributions with binomial moments.
\newblock {\em Infin. Dimens. Anal. Quantum Probab. Relat. Top.},
  17(2):1450014, 32, 2014.

\bibitem[MP18]{Mlotowski2018}
Wojciech M{\l}otkowski and Karol~A. Penson.
\newblock A {F}uss-type family of positive definite sequences.
\newblock {\em Colloq. Math.}, 151(2):289--304, 2018.

\bibitem[NS06]{NicaSpeicher}
Alexandru Nica and Roland Speicher.
\newblock {\em Lectures on the combinatorics of free probability}, volume 335
  of {\em London Mathematical Society Lecture Note Series}.
\newblock Cambridge University Press, Cambridge, 2006.

\bibitem[OEI]{OEIS}
{OEIS} {F}oundation {I}nc. (2018), {T}he {O}n-{L}ine {E}ncyclopedia of
  {I}nteger {S}equences.

\bibitem[Pos]{postnikov-webs}
Alex Postnikov.
\newblock Total positivity, grassmannians, and networks.
\newblock arXiv:math/0609764.

\bibitem[Rad79]{Radoux1979}
Christian Radoux.
\newblock Calcul effectif de certains d\'{e}terminants de {H}ankel.
\newblock {\em Bull. Soc. Math. Belg. S\'{e}r. B}, 31(1):49--55, 1979.

\bibitem[Rad90]{Radoux1990}
Christian Radoux.
\newblock D\'{e}terminant de {H}ankel construit sur les polyn\^{o}mes de
  {H}ermite.
\newblock {\em Ann. Soc. Sci. Bruxelles S\'{e}r. I}, 104(2):59--61 (1991),
  1990.

\bibitem[Rad91]{Radoux1991}
Christian Radoux.
\newblock D\'{e}terminant de {H}ankel construit sur des polyn\^{o}mes li\'{e}s
  aux nombres de d\'{e}rangements.
\newblock {\em European J. Combin.}, 12(4):327--329, 1991.

\bibitem[Rad92]{RadouxSLC}
Christian Radoux.
\newblock D\'{e}terminants de {H}ankel et th\'{e}or\`eme de {S}ylvester.
\newblock In {\em S\'{e}minaire {L}otharingien de {C}ombinatoire
  ({S}aint-{N}abor, 1992)}, volume 498 of {\em Publ. Inst. Rech. Math. Av.},
  pages 115--122. Univ. Louis Pasteur, Strasbourg, 1992.

\bibitem[Rad00]{Radoux2000}
Christian Radoux.
\newblock Addition formulas for polynomials built on classical combinatorial
  sequences.
\newblock In {\em Proceedings of the 8th {I}nternational {C}ongress on
  {C}omputational and {A}pplied {M}athematics, {ICCAM}-98 ({L}euven)}, volume
  115, pages 471--477, 2000.

\bibitem[Rad02]{Radoux2002}
Christian Radoux.
\newblock The {H}ankel determinant of exponential polynomials: A very short
  proof and a new result concerning {E}uler numbers.
\newblock {\em Amer. Math. Monthly}, 109(3):277--278, 2002.

\bibitem[Ran98]{randrianarivony-poly-ortho-sheffer}
Arthur Randrianarivony.
\newblock Moments des polyn\^{o}mes orthogonaux unitaires de {S}heffer
  g\'{e}n\'{e}ralis\'{e}s et sp\'{e}cialisations.
\newblock {\em European J. Combin.}, 19(4):507--518, 1998.

\bibitem[She53]{shephard-unitary-groups}
G.~C. Shephard.
\newblock Unitary groups generated by reflections.
\newblock {\em Canad. J. Math.}, 5:364--383, 1953.

\bibitem[Sok18]{Sokal2018}
Alan~D. Sokal.
\newblock The {E}uler and {S}pringer numbers as moment sequences.
\newblock {\em Expo. Math.}, in press, 2018.

\bibitem[Spe92]{Speicher1992}
Roland Speicher.
\newblock A noncommutative central limit theorem.
\newblock {\em Math. Z.}, 209(1):55--66, 1992.

\bibitem[SS85]{simion-schmidt}
Rodica Simion and Frank~W. Schmidt.
\newblock Restricted permutations.
\newblock {\em European J. Combin.}, 6(4):383--406, 1985.

\bibitem[SS94]{simion-stanton-laguerre}
R.~Simion and D.~Stanton.
\newblock Specializations of generalized {L}aguerre polynomials.
\newblock {\em SIAM J. Math. Anal.}, 25(2):712--719, 1994.

\bibitem[SS96]{simion-stanton-octabasic}
R.~Simion and D.~Stanton.
\newblock Octabasic {L}aguerre polynomials and permutation statistics.
\newblock {\em J. Comput. Appl. Math.}, 68(1-2):297--329, 1996.

\bibitem[ST43]{Shohat}
J.~A. Shohat and J.~D. Tamarkin.
\newblock {\em The {P}roblem of {M}oments}.
\newblock American Mathematical Society Mathematical surveys, vol. I. American
  Mathematical Society, New York, 1943.

\bibitem[Ste94]{es-indexed}
Einar Steingr{\'{\i}}msson.
\newblock Permutation statistics of indexed permutations.
\newblock {\em European J. Combin.}, 15(2):187--205, 1994.

\bibitem[SW06]{savage-wilf-patts-comp-multisets}
Carla~D. Savage and Herbert~S. Wilf.
\newblock Pattern avoidance in compositions and multiset permutations.
\newblock {\em Adv. in Appl. Math.}, 36(2):194--201, 2006.

\bibitem[SW07a]{shareshian-wachs-exc-maj}
John Shareshian and Michelle~L. Wachs.
\newblock {$q$}-{E}ulerian polynomials: excedance number and major index.
\newblock {\em Electron. Res. Announc. Amer. Math. Soc.}, 13:33--45, 2007.

\bibitem[SW07b]{stein-williams}
Einar Steingr{\'{\i}}msson and Lauren~K. Williams.
\newblock Permutation tableaux and permutation patterns.
\newblock {\em J. Combin. Theory Ser. A}, 114(2):211--234, 2007.

\bibitem[SZ]{sokal-zeng-masterpolys}
Alan~D. Sokal and Jiang Zeng.
\newblock Some multivariate master polynomials for permutations, set
  partitions, and perfect matchings, and their continued fractions.
\newblock arXiv:2003.08192.

\bibitem[Wig55]{Wigner1955}
Eugene~P. Wigner.
\newblock Characteristic vectors of bordered matrices with infinite dimensions.
\newblock {\em Ann. of Math. (2)}, 62:548--564, 1955.

\bibitem[Wim00]{Wimp2000}
Jet Wimp.
\newblock Hankel determinants of some polynomials arising in combinatorial
  analysis.
\newblock volume~24, pages 179--193. 2000.
\newblock Computational methods from rational approximation theory (Wilrijk,
  1999).

\end{thebibliography}

\end{document}